\documentclass[11pt]{amsart}
\usepackage{amsmath,amsfonts,amsthm,bbm, amssymb}
\usepackage[cmtip, all]{xy}
\usepackage[letterpaper, left=3cm, right=3cm, top = 2.5cm, bottom = 2.5cm ]{geometry}
\usepackage{comment}

\newtheorem{thm}{Theorem}[section]
\newtheorem{prop}[thm]{Proposition}
\newtheorem{lem}[thm]{Lemma} 
\newtheorem{cor}[thm]{Corollary}

\newtheorem{ques}[thm]{Question}

\theoremstyle{definition}

\newtheorem{defn}[thm]{Definition}
\theoremstyle{remark}
\newtheorem{remk}[thm]{Remark}
\newtheorem{remks}[thm]{Remarks}

\newtheorem{exm}[thm]{Example}
\newtheorem{exms}[thm]{Examples}
\newtheorem{notat}[thm]{Notation}
\numberwithin{equation}{section}

{\hfill$\square$\end{defn}}
{\hfill$\square$\end{remk}}
{\hfill$\square$\end{remks}}
{\hfill$\square$\end{exm}}
{\hfill$\square$\end{exms}}
{\hfill$\square$\end{notat}}

\newcommand{\thmref}{Theorem~\ref}
\newcommand{\propref}{Proposition~\ref}
\newcommand{\corref}{Corollary~\ref}

\newcommand{\lemref}{Lemma~\ref}

\newcommand{\sC}{{\mathcal C}}

\newcommand{\sF}{{\mathcal F}}

\newcommand{\sI}{{\mathcal I}}
\newcommand{\sJ}{{\mathcal J}}
\newcommand{\sK}{{\mathcal K}}
\newcommand{\sL}{{\mathcal L}}

\newcommand{\sO}{{\mathcal O}}

\newcommand{\sR}{{\mathcal R}}

\newcommand{\sU}{{\mathcal U}}

\newcommand{\sZ}{{\mathcal Z}}

\newcommand{\A}{{\mathbb A}}

\newcommand{\F}{{\mathbb F}}
\newcommand{\G}{{\mathbb G}}

\newcommand{\N}{{\mathbb N}}
\renewcommand{\P}{{\mathbb P}}
\newcommand{\Q}{{\mathbb Q}}

\newcommand{\Z}{{\mathbb Z}}

\newcommand{\fm}{{\mathfrak m}}

\newcommand{\fp}{{\mathfrak p}}

\newcommand{\Alb}{{\rm Alb}}
\newcommand{\CH}{{\rm CH}}

\newcommand{\surj}{\twoheadrightarrow}
\newcommand{\inj}{\hookrightarrow}

\newcommand{\Pic}{{\rm Pic}}

\newcommand{\divf}{{\rm div}}
\newcommand{\Hom}{{\rm Hom}}

\newcommand{\Spec}{{\rm Spec \,}}
\newcommand{\sing}{{\rm sing}}

\newcommand{\Sch}{{\operatorname{\mathbf{Sch}}}}

\newcommand{\Sm}{{\mathbf{Sm}}}

\newcommand{\et}{{\text{\'et}}}

\newcommand{\ds}{{/\kern-3pt/}}

\newcommand{\ov}{\overline}

\renewcommand{\dim}{\text{\rm dim}}

\newcommand{\tuborg}{\left\{\begin{array}{ll}}
\newcommand{\sluttuborg}{\end{array}\right.}

\newcommand{\wt}{\widetilde}

\newcounter{elno}

\newcounter{elno-abc}   

\newcounter{elno-abc-prime}

\begin{document}\title{Murthy's conjecture on 0-cycles}
\author{Amalendu Krishna}
\address{School of Mathematics, Tata Institute of Fundamental Research,  
1 Homi Bhabha Road, Colaba, Mumbai, India}
\email{amal@math.tifr.res.in}




\begin{abstract}
We show that the Levine-Weibel Chow group of 0-cycles $\CH^d(A)$ of a reduced
affine algebra $A$ 
of dimension $d \ge 2$ over an algebraically closed field is torsion-free.
Among several
applications, it implies an affirmative solution to an old conjecture of
Murthy in classical $K$-theory.
\end{abstract}
\setcounter{tocdepth}{1}
\maketitle
\tableofcontents

\section{Introduction}\label{sec:Intro}
A very classical question in the study of projective modules over
commutative Noetherian rings is to determine conditions for a finitely generated
projective module to split off a free summand of positive rank.
In geometric terms, it means to determine conditions 
for a vector bundle on an affine scheme to admit a nowhere 
vanishing section.  A well known result of Serre \cite{Serre} says that 
if a finitely generated projective module over a commutative Noetherian ring 
$A$ has rank $r > \dim(A)$, then  it splits off a free summand of positive rank.
However, this question becomes very subtle when $r \le \dim(A)$.
For affine algebras over algebraically closed fields, this question was
studied by Murthy in his seminal paper \cite{Murthy}.

\subsection{Murthy's conjecture}
Let $k$ be an algebraically closed field and let $A$ be a
reduced affine algebra of dimension $d \ge 1$.
Let $K_0(A)$ denote the Grothendieck group of finitely generated 
projective $A$-modules. A projective $A$-module $P$ of rank $d$ admits a Chern
class $c_d(P) = \stackrel{d}{\underset{i = 0}\sum}
(-1)^i [\wedge^i(P^{\vee})] \in K_0(A)$. 
If $A$ is smooth, $c_d(P)$ maps to the top Chern class of $P$
in the Chow group $\CH^d(A)$ via the 
Chern class map $c_d:K_0(A) \to \CH^d(A)$, constructed by Grothendieck.
Recall that every smooth maximal ideal of height $d$ in $A$ has a class in 
$K_0(A)$ and $F^dK_0(A)$ is the subgroup of $K_0(A)$ generated by these classes.
In order to answer the above question for projective $A$-modules of
rank $d$, Murthy posed the following in \cite[Question~2.12]{Murthy}. 

\begin{ques}$(${\rm Murthy}$)$\label{question:Murthy-Open}
Let $A$ be a reduced affine algebra of dimension $d \ge 2$ over an 
algebraically closed field $k$. Is $F^dK_0(A)$ torsion-free?
\end{ques} 

This question has acquired great significance in the study of projective modules
because assuming its positive answer, Murthy already deduced
a series of outstanding results, one of which solves the 
splitting problem in `rank = dimension' case as follows.

\begin{thm}$($\cite[Theorem~3.8]{Murthy}$)$\label{thm:Intro-Res-3*}
Assume that Question~\ref{question:Murthy-Open} has a positive solution.
Let $A$ be a reduced affine algebra of dimension $d \ge 1$ 
over an algebraically closed field $k$.
Let $P$ be a projective $A$-module of rank $d$. Then $P$ splits off 
a free summand of positive rank if and only if $c_d(P) = 0$ in $K_0(A)$.
\end{thm}

It turns out that Murthy's question is closely related to another
open question in the theory of Chow group of
0-cycles on affine schemes. In this paper, we solve this open question
and derive affirmative answer to Murthy's question as several of its 
consequences.

\subsection{Affine Roitman torsion problem}\label{sec:Approach}
Let $k$ be an algebraically closed field.
Given a quasi-projective scheme $X$ of dimension $d$ over $k$, let
$\CH^d(X)$ (see \S~\ref{sec:LWC}) denote the Levine-Weibel Chow group of 
0-cycles on $X$ \cite{LW}. This is generated by the classes of 
regular closed points on $X$ and coincides with
the classical definition of the Chow group of 0-cycles when $X$ is smooth. 
If $X = \Spec(A)$ is affine, we also write it as $\CH^d(A)$. 
The {\sl affine Roitman torsion problem} asks whether 
$\CH^d(X)$ is torsion-free when $X$ is an affine variety of
dimension $d \ge 2$ over $k$. Note that $\CH^1(X)$ need not be torsion-free
for a smooth affine curve $X$.

It is a consequence of the torsion theorems of Roitman \cite{Roitman},
Milne \cite{Milne} and Krishna-Srinivas \cite{KSri} 
that the affine Roitman torsion problem has a positive 
solution if $X$ is normal.
If $k$ has exponential characteristic $p \ge 1$, the affine Roitman torsion 
problem for singular affine varieties was affirmatively 
solved by Levine \cite{Levine-2}, modulo $p$-primary torsion.
If $X$ is a normal affine variety of dimension $d \ge 3$
in positive characteristic, affirmative
solutions were found by Srinivas \cite{Srinivas}.
The works of Levine and Srinivas predated \cite{KSri}.

In this paper, we solve the general case
of the affine Roitman torsion problem as follows (see
\thmref{thm:Tor-High-D}).

\begin{thm}\label{thm:Intro-Res-1}
For any reduced affine algebra $A$ of dimension $d \ge 2$ over an 
algebraically closed field, the Levine-Weibel Chow group of 0-cycles 
$\CH^d(A)$ is uniquely divisible.
\end{thm}

\vskip .3cm

\subsection{Resolution of Murthy's conjecture}
\thmref{thm:Intro-Res-1} does not immediately 
resolve Murthy's conjecture.
However, one knows that there is a cycle class map
$cyc_A \colon \CH^d(A) \to K_0(A)$ (see \cite[Proposition~2.1]{LW}) and
$F^dK_0(A)$ is evidently its image. Murthy's question will therefore be
positively answered by  \thmref{thm:Intro-Res-1}
if we can show additionally that the kernel of this cycle class map
is a torsion group.  But this follows from the following 
stronger result.

\begin{thm}\label{thm:Intro-RR}
Let $A$ be a reduced affine algebra of dimension $d \ge 1$ over an 
algebraically closed field. 
Then ${\rm Ker}(cyc_A)$ is a torsion group of exponent $(d-1)!$.
\end{thm}

We give a proof of this theorem in \S~\ref{sec:ECG} (see \thmref{thm:RR}). 
A more general result, which holds over arbitrary fields,
is obtained in \cite[Theorem~7.5]{GK}.
For smooth affine
algebras, this result is classical (see \cite[\S~4.3]{Grothendieck}). 
For singular algebras, a proof is given in an old unpublished 
manuscript \cite{Levine-S} of Levine. 
The dimension two case of Levine's proof is available in
\cite{BS-*}. For normal surfaces, it is also shown in \cite{PW2}.

\vskip .3cm

Combining Theorems~\ref{thm:Intro-Res-1} and ~\ref{thm:Intro-RR}, we
conclude the following stronger form of Murthy's conjecture
(see \corref{cor:Murthy*}).

\begin{cor}\label{cor:Intro-Res-1*}
Let $A$ be a reduced affine algebra of dimension $d \ge 2$ over an 
algebraically closed field.  Then $F^dK_0(A)$ is uniquely divisible.
\end{cor}

Like the affine Roitman torsion problem, Murthy's question was also studied by
other authors in the past. For normal surfaces,
it follows from \cite{KSri} and \cite{PW2}.
Levine \cite{Levine-2} proved its 
prime-to-characteristic (in particular, the characteristic zero) part.
An affirmative answer was given by Srinivas \cite{Srinivas} for
normal affine varieties in dimension $d \ge 3$. 
For normal affine varieties in all dimensions, a positive solution was given
in \cite{KSri}. 

However, all of these results (\cite{Levine-2}, \cite{Srinivas} and
\cite{KSri})
were conditional on \thmref{thm:Intro-RR} (or the unpublished
work of Levine).
As far as we are aware, except for normal surfaces 
(and all surfaces in characteristic zero), \corref{cor:Intro-Res-1*} gives 
first unconditional solution of Murthy's conjecture in any other case
(including the case of characteristic zero or the case of normal varieties).



\vskip .3cm

Apart from the solution to Murthy's question, \thmref{thm:Intro-Res-1} and
the method of its proof have many other remarkable consequences.
In this paper, we give the following applications to Euler class groups and 
Bloch-Srinivas conjecture.

\subsection{The Euler class group}\label{sec:ECG-Main}
If the base field $k$ is not algebraically closed, \thmref{thm:Intro-Res-3*} no 
longer holds even for smooth affine varieties, as the famous example of the
tangent bundle on the real 2-sphere shows. To remedy this, Nori defined a
finer invariant than the Chow group, namely, 
the `Euler class group' $E(A)$ of a smooth affine algebra $A$ over any
field. Later, the notion of Euler class group was defined by 
Bhatwadekar and Raja Sridharan \cite{BS-3} for any commutative Noetherian ring.
This group admits the Euler class $e(P)$ of any projective module $P$ of
rank $= \dim(A)$ with trivial determinant. 

It was shown in \cite{BS-1} and \cite{BS-3} that 
for $A$ either smooth or containing $\Q$, the vanishing of $e(P)$ is 
a necessary and sufficient condition for $P$ to split off a free summand
of positive rank. 
\thmref{thm:Intro-Res-3*} therefore suggests that the Euler class group
of \cite{BS-3} should coincide with $\CH^d(A)$ for any reduced
affine algebra $A$ over $k$, if it is algebraically closed.
As another application of \thmref{thm:Intro-Res-1}, we show in 
this paper (see \thmref{thm:EGG-Chow}) that this
is indeed the case. When $A$ is smooth, this was conjectured
in \cite[Remark~3.3]{BS-4} and proven in \cite[Corollary~4.15]{BS-1}.

\begin{thm}\label{thm:Intro-Res-4}
Let $A$ be a reduced affine algebra of dimension $d \ge 2$ over an 
algebraically closed field.
Then there is a canonical isomorphism $E(A) \xrightarrow{\simeq} \CH^d(A)$.
\end{thm}


One consequence of Theorem~\ref{thm:Intro-Res-4}  is that
together with  
Murthy's Chern classes $c_d(P) \in F^dK_0(A)$ and 
\thmref{thm:Intro-RR}, it solves the problem of defining
the Euler classes of rank $d$ projective $A$-modules  
in $E(A)$ without any condition on $A$. The vanishing of these Euler classes 
determines the splitting of the projective modules, as in
\thmref{thm:Intro-Res-3*}.

M. Schlichting \cite{Schl} has recently defined Euler classes of 
projective modules of top rank over an affine algebra in a cohomology group of 
some Milnor $K$-theory 
sheaf and has proven an analogue of \thmref{thm:Intro-Res-3*} for his 
Euler classes. Using \thmref{thm:Intro-Res-1}, it is shown in 
\cite{GK} that Schlichting's cohomology group coincides with the
Levine-Weibel Chow group (see below). This was known before in
dimension two using \cite{Levine-1} and \cite{BS-*}.
However, we do not know if the two Euler classes are comparable. The reason for
the difficulty is that unlike the class $c_d(P)$, we do not know how 
to explicitly describe Schlichting's Euler class.

\vskip .3cm


\subsection{A singular analogue of Beilinson-Bloch conjecture}
\label{sec:BBC}
To see another important consequence of \thmref{thm:Intro-Res-4}, recall
that in the study of geometric properties of commutative algebras,
one often needs to know if a given ideal in such an algebra
is a complete intersection. In this context, the Beilinson-Bloch 
conjecture motivates the following.

\begin{ques}\label{ques:CI}
Let $k$ be either $\ov{\F}_p$ or $\ov{\Q}$ and let  
$A$ be a reduced affine algebra of dimension $d \ge 2$ over $k$.
Is every local complete intersection ideal of height $d$ in 
$A$ a complete intersection?
\end{ques}

If $k = \ov{\F}_p$ and $A$ is regular, this question has a positive
solution. 
If $k = \ov{\Q}$ and $A$ is regular of dimension $d = 2$, then the above
question is simply an algebraic formulation of a
famous conjecture and Beilinson and Bloch. This conjecture is
still open except in some very special cases.
Using \thmref{thm:Intro-Res-4}, we can solve this question in the
following cases (see \S~\ref{sec:BBC*}).

\begin{cor}\label{cor:Intro-Res-5}
Let $A$ be a reduced affine algebra of dimension $d \ge 2$ over an 
algebraically closed field $k$. Assume that one of the following holds.
\begin{enumerate}
\item
$k = \ov{\F}_p$.
\item
$A = \oplus_{i \ge 0} A_i$ is a graded $k$-algebra with 
$A_0 = k = \ov{\Q}$.
\end{enumerate}

Then every local complete intersection ideal of height $d$ in 
$A$ is a complete intersection.
\end{cor}

A weaker version of \corref{cor:Intro-Res-5} was shown in
\cite[Theorems~6.4.1, 6.4.2]{KSri-1}, namely, that every smooth maximal ideal 
of height $d$ in $A$ is a complete intersection.

\subsection{Strong Bloch-Srinivas conjecture}\label{sec:BSC}
Let $X$ be a reduced quasi-projective scheme of dimension $d$ with 
isolated Cohen-Macaulay singularities over a perfect field $k$.
Given a resolution of singularities $\pi \colon {\wt X} \to X$,
the term {\sl reduced exceptional divisor} will mean the reduced part of
the actual exceptional divisor (which may itself may not be reduced) of $\pi$.
 
Let $F^{d}K_0(X)$ denote the subgroup of the Grothendieck group of vector 
bundles $K_0(X)$, generated by the classes of smooth codimension $d$ 
points on $X$. 
For any $Z \subsetneq X$ closed, let $rZ$ denotes the $r^{\rm th}$ 
infinitesimal thickening of $Z$ in $X$. Let $F^{d}K_{0}(X, rZ)$ be the subgroup 
of the relative $K$-group $K_{0}(X, rZ)$, generated by the classes of 
smooth codimension $d$ points of $X \setminus Z$ (see \cite[\S~2]{KSri}).

The Bloch-Srinivas conjecture (originally stated only for surfaces in
\cite{Srinivas-1}) predicts that if there exists a resolution of singularities 
$\pi \colon {\wt X} \to X$ with reduced exceptional divisor 
$E \subset \wt X$, then
the pull-back map $\pi^*\colon F^dK_0(X) \to F^dK_0(\wt{X})$ induces an
isomorphism 
\begin{equation}\label{eqn:BS-form}
F^dK_0(X) \xrightarrow{\simeq} {\underset{r}\varprojlim}\
F^dK_0(\wt{X}, rE).
\end{equation}

The $d =2$ case of this conjecture was proven in \cite{KSri}.
The $d \ge 3$ case in characteristic zero was proven in
\cite{Krishna-4}. Morrow \cite{Morrow-1} has recently proven a 
version of the conjecture for non-Cohen-Macaulay singularities 
in characteristic zero.
In general, it is known that one can not replace the inverse limit in
~\eqref{eqn:BS-form} by $F^dK_0(\wt{X}, E)$. However, using the method of
the proof of \thmref{thm:Intro-Res-1}, we obtain the
following surprisingly strong version of the Bloch-Srinivas conjecture in
positive characteristic (see \S~\ref{sec:Pf-th-1.7}).

\begin{thm}\label{thm:Chowgroup}
Let $X$ be an affine or a projective variety of dimension $d \ge 2$ over an
algebraically closed field $k$ of positive characteristic
such that $X$ has only isolated Cohen-Macaulay singularities.
Suppose there is a resolution of singularities $\pi \colon {\wt X} \to X$
with the reduced exceptional divisor $E \subset \wt X$. Then 
there are isomorphisms
\[
\CH^d(X) \xrightarrow{\simeq} F^{d}K_{0}(X) \xrightarrow{\simeq} 
F^{d}K_{0}(\wt {X},E).
\]
\end{thm}	

\thmref{thm:Chowgroup} has the following consequence for the Chow
group of 0-cycles and vector bundles on the affine cone of a
non-singular closed subscheme of $\P^n_k$ (see \S~\ref{sec:SBC-**}).

\begin{cor}\label{cor:BSC-0-1}
Let $k$ be an algebraically closed field of characteristic $p > 0$.
Let $A$ be the homogeneous coordinate ring of a smooth projective
variety $Z \inj \P^n_k$ of dimension $d-1 \ge 1$. Assume that $A$
is Cohen-Macaulay. Then $\CH^{d}(A) = 0$. In particular,
every projective $A$-module of rank $d$ splits off 
a free summand of positive rank and, every local complete intersection ideal of 
height $d$ in $A$ is a complete intersection.
\end{cor}

Note that if $A$ is normal, there is a cohomological criterion for $A$ to be 
Cohen-Macaulay (see \cite[Ex.~18.16, p. 468]{Eisenbud}).
This is given in terms of the vanishing of 
$H^i(Z, \sO_Z(n))$ for $1 \le i \le d-2$
and $n \ge 0$. In particular, $A$ is Cohen-Macaulay if it is normal and
$Z \inj \P^n_k$ is a complete intersection.

We remark that \corref{cor:BSC-0-1} is very specific to characteristic $p > 0$
as there are counter-examples in characteristic zero 
(see, for instance, \cite[Corollary~1.4]{KSri}). For a modified form 
of \corref{cor:BSC-0-1}  in characteristic zero, see 
\cite[Theorem~1.5]{Krishna-4}.

Using \cite[Theorem~1.5]{BK}, another immediate consequence of 
\thmref{thm:Chowgroup} is the following result
about the Chow group of 0-cycles with modulus introduced by Kerz-Saito \cite{KS}
(see below). This result is significant because the Chow groups with
modulus in general behave well only when considered as 
pro-abelian groups $\{\CH^d(X, mD)\}_{m \ge 1}$ (see, for example,
the main results of \cite{KS}). 

\begin{cor}\label{cor:modulus-Chow-grp}
Assuming the notations and hypotheses of \thmref{thm:Chowgroup},
the restriction map $\CH^d(\wt{X}, mE) \to \CH^d(\wt{X}, E)$ is an isomorphism 
for all $m \ge 1$. In particular, the pro-abelian group 
$\{\CH^d(\wt{X}, mE)\}_{m \ge 1}$ is constant. 
\end{cor}

\vskip .3cm

\subsection{Further applications}\label{sec:modulus}
Many other outstanding applications of the affine Roitman
torsion theorem and Murthy's conjecture have been obtained recently. 
Applications to torsion in the Chow groups with modulus 
were obtained in \cite[Theorem~1.4]{BK} and \cite{Krishna-5}.
In particular, this yields Roitman torsion theorem for the Chow group
of 0-cycles with modulus.
Bloch's formula for the Levine-Weibel Chow group and the
0-cycle group with modulus was shown in \cite{GK}.
An application of Murthy's conjecture in the identification of a 
motivic spectral sequence for the relative $K$-theory was obtained in \cite{KP}.
In a very recent work \cite{K-Sarwar}, Murthy's conjecture was used to
explicitly describe the Levine-Weibel Chow group of monoid algebras
(see \cite{CHW-1}, \cite{CHW-2}, \cite{CHW-3}, \cite{Gub-7}
and \cite{K-Sarwar-1}). In particular,
it was shown that the Levine-Weibel Chow group of an affine toric variety
vanishes.

\vskip .4cm

\subsection{Outline of the proofs}
We now give an outline of our main proofs. As we mentioned above,
for torsion prime to the characteristic of the ground field $k$,
\thmref{thm:Intro-Res-1} was proven by Levine \cite{Levine-2}
many years ago. The key point of his proof is to do the following.
\begin{enumerate}
\item
Given a reduced affine scheme $X$ of dimension $d \ge 2$ 
and an integer $n$ prime to ${\rm char}(k)$, construct a map  
$n^{-1}_X\colon \CH^d(X) \to \CH^d(X)$, which is inverse to the multiplication
by $n$ on $\CH^d(X)$.  
\end{enumerate}

The idea of this construction is to first note that any 0-cycle on $X$
lies on a `good curve' $C$ on (or a composition of monoidal transformations
of) $X$ and then use the fact that $\Pic(C)$ is $n$-divisible.  
To show that this process makes $n^{-1}_X$ well defined, one relies on
the following two facts.

\hspace*{.1cm} (2) \ The map $_{n}\Pic(C) \to \CH^d(X)$ is zero, where 
$_{n}\Pic(C)$ is the group of $n$-torsion points \\
\hspace*{1.3cm} in $\Pic(C)(k)$.


\hspace*{.1cm} (3) \ A rigidity statement, which says that if there is a 
family of  
$n$-torsion 0-cycles in \\
\hspace*{1.3cm} $\CH^d(X)$ parameterized by a smooth curve,
then this family must be constant.

As mentioned in (1), a fundamental fact which is critically used 
throughout the proof of each of the 
three steps above is that $\Pic(C)(k)$ is an $n$-divisible group
for any reduced curve $C$. Apart from the above steps, this fact is
used repeatedly in almost every intermediate result of \cite{Levine-2}.

One knows that in characteristic $p > 0$, the Picard group of a reduced
curve often has a non-trivial $p$-primary torsion subgroup, namely, the
unipotent part.
As a result, the above idea does not generalize to study the $p$-torsion
in $\CH^d(X)$.

The ideas and approach of our proofs are conceptually different from 
\cite{Levine-2}. Levine does not prove his final result by  
induction on $d$ (though he reduces the proofs of most of his
intermediate steps to $d =2$ case). 
In this paper, we first prove the main result for surfaces.
Here, contrary to relying on the divisibility of the unipotent part
of $\Pic(C)$, we exploit the fact that it is torsion. 
What helps us is the observation that such a torsion has a bounded exponent.
More generally, we need to show (see \S~\ref{sec:pro-desc}) that 
the bi-relative $K$-theory associated to finite and some other abstract blow-ups
in positive characteristic are torsion of bounded exponents.
The proof uses some
outstanding results of Geisser and Hesselholt \cite{GH-2} and \cite{GH-1}.

In the next step (see Sections~\ref{sec:nor-proj}  and ~\ref{sec:Non-affine}), 
we establish a divisibility property of the $SK_1$ of
certain normal surfaces. Apart from the bound on the $p$-torsion,
this also requires some vanishing results of Geisser and Levine \cite{GL}
for mod-$p$ $K$-theory of smooth schemes, 
and a precise computation of the torsion in the $SK_1$ of a curve.
A combination of these and a Mayer-Vietoris sequence
yields the proof of the torsion theorem for surfaces.

Our proof of \thmref{thm:Intro-Res-1} in higher dimension is by reduction
to the case of surfaces. This goes through several steps
(see \S~\ref{sec:High-D}) and is similar in spirit to some reduction steps
of Levine (see \cite[Lemmas~2.1, 2.3]{Levine-2}).
The basic idea is the following. 

Given a 0-cycle $\alpha$ on $X$ (assuming
$\dim(X) \ge 3$) such that $n\alpha$ is the divisor 
of a rational function on a  reduced 
`Cartier curve' (see \S~\ref{sec:Zero-C}) $C \subset X$, 
the strategy is to embed $C$ in a nice enough surface which maps to $X$.
However, we can not do this in general
due to the possibility that $C$ may have high embedding dimension. 
So we use Bloch's trick of blowing up $X$ enough number of times 
along its smooth points, lying in $C$, so that
the embedding dimension of the strict transform of $C$ becomes at most two.

At this stage, we use Bertini theorems (see \cite[Theorem~7]{KL} and
\cite[Lemma~1.3]{Levine-2})
to find a complete intersection surface $Y'$ in the blow up $X'$ which contains
the strict transform of $C$ and, which is regular over the regular locus 
of $X$. In particular, the cycle $\alpha$ lifts to a cycle $\alpha'$ on
$Y'$ such that $n\alpha'$ dies in $\CH^2(Y')$. Using the
torsion result for surfaces (here we need to prove our theorem 
for some non-affine surfaces, see 
\thmref{thm:Main-affine}), we show that $\alpha'$ dies in $\CH^2(Y')$.
In the final step, we use the push-forward
map between the Chow groups of 0-cycles to kill $\alpha$ in $\CH^d(X)$.

Applications of \thmref{thm:Intro-Res-1} (and its proof) to Murthy's question, 
the Euler class groups and
the Bloch-Srinivas conjecture in positive characteristic 
are given in Sections ~\ref{sec:ECG} and ~\ref{sec:Pf:Chow-res}.

\vskip .4cm

\subsection{Notations} Let $k$ be a field and let $\Sch_k$ denote the 
category of separated schemes of finite type over $k$. 
We shall let $\Sm_k$ denote the category
of those schemes in $\Sch_k$ which are smooth over $k$. \
For $X \in \Sch_k$, the normalization of $X_{\rm red}$ will be denoted by
$X^N$.
Given a closed immersion $Y \subset X$ in $\Sch_k$ and a positive integer
$r$, the scheme $rY \subset X$ will denote the closed subscheme of
$X$ defined by the sheaf of ideals $\sI^r_Y$, where $\sI_Y$ is
the sheaf of ideals in $\sO_X$ defining $Y$.
We shall call this the $r$-th infinitesimal thickening of $Y$ inside $X$.
We shall specify the nature of the field $k$ in each section.

For any abelian group $A$ and a prime number $l$, we shall
denote $A \otimes_{\Z} {\Q_l}/{\Z_l} = {\underset{m}\varinjlim}\ {A}/{l^m}$ by 
${A}/{l^{\infty}}$. The $l$-primary torsion subgroup of $A$ will be
denoted by ${_{l^{\infty}}A}$.

\section{Review of 0-cycles and 
some preliminary results}\label{sec:LWC}
In this section, we recall the definition of the Chow group
of 0-cycles on singular schemes from \cite{LW}. 
We prove some properties of this Chow group which will be used in this
paper. We also prove some other preliminary results we need for our
torsion theorem. Unless further assumptions are specified, 
$k$ will denote a perfect field in this section.

\subsection{0-cycles on singular schemes}\label{sec:Zero-C}
Let $X$ be a reduced quasi-projective scheme over $k$ of dimension $d \ge 1$. 
Let $X_{\rm reg}$ denote the disjoint union of the smooth loci of the 
$d$-dimensional irreducible components of $X$.
A regular (or smooth) closed point of $X$ will mean a closed point lying in
$X_{\rm reg}$.
Let $X_{\rm sing}$ denote the complement of $X_{\rm reg}$ in $X$ with the
reduced induced closed subscheme structure.
Let $Y \subsetneq X$ be a closed subset not containing any $d$-dimensional
component of $X$ such that $X_{\rm sing} \subseteq Y$. 
Let $\sZ^d(X,Y)$ be the free abelian group on closed points of $X \setminus Y$.
We shall often write $\sZ^d(X,X_{\rm sing})$ as $\sZ^d(X)$.
A (reduced) \emph{Cartier curve on $X$ relative to $Y$} is a purely 
$1$-dimensional closed subscheme $C\hookrightarrow X$ that is reduced, has 
no component contained in $Y$ and is defined by a regular sequence in $X$ at 
each point of $C\cap Y$.

Let $C$ be a Cartier curve in $X$ relative to $Y$ and let
$\{\eta_1, \cdots , \eta_r\}$ denote the set of its generic points.
Let $\sO_{C, C \cap Y}$ denote the semilocal ring of $C$ at
$(C \cap Y) \cup \{\eta_1, \cdots , \eta_r\}$.
Let $k(C)$ denote the ring of total quotients of $C$.
Notice that $\sO_{C, C \cap Y}$ and $k(C)$ coincide if $C \cap Y = \emptyset$.
Since $C$ is a reduced curve, it is Cohen-Macaulay and hence the 
canonical map
$k(C) \to \stackrel{r}{\underset{i = 1}\prod}  \sO_{C, \eta_i}$
is an isomorphism. In particular, the map 
$\theta_C\colon\sO^{\times}_{C, C \cap Y} \to  
\stackrel{r}{\underset{i = 1}\prod}  \sO^{\times}_{C, \eta_i}$ is injective. 

Given $f \in \sO^{\times}_{C, C \cap Y}$, let $\{f_i\} = \theta_C(f)$ and let
$(f_i)_{\eta_i}:= {\rm div}(f_i)$ denote the divisor of zeros and poles of $f_i$ 
on $\ov{\{\eta_i\}}$ in the sense of \cite{Fulton}. We let $(f)_C:=
\stackrel{r}{\underset{i =1}\sum} (f_i)_{\eta_i}$. As $f$ is an invertible
regular function on $C$ in a neighborhood of $C \cap Y$, 
we see that $(f)_C \in \sZ^d(X,Y)$.

Let $\sR^d(X,Y)$ denote the subgroup of $\sZ^d(X,Y)$ generated by
$(f)_C$, where $C$ is a Cartier curve on $X$ relative to $Y$ and
$f \in \sO^{\times}_{C, C \cap Y}$. The Chow group of 0-cycles on $X$ relative to
$Y$ is the quotient
\begin{equation}\label{eqn:0-cyc-1}
\CH^d(X,Y) = \frac{\sZ^d(X,Y)}{\sR^d(X,Y)}.
\end{equation}

If $X$ is a projective scheme of dimension $d$ over $k$,
the push-forward on the 0-cycles defines the degree map
${\rm deg}\colon \CH^d(X, Y) \to \Z^r$, where $r$ is the number of
irreducible components of $X$. We let $\CH^d(X,Y)_0$ denote the kernel
of this map. 

The group $\CH^d(X, X_{\rm sing})$ is denoted in short by $\CH^d(X)$, and is 
called the {\sl Chow group of 0-cycles on $X$}.
Bertini type theorems for singular schemes (see \cite[\S~1]{Levine-2})
can be used to prove the following expression for the elements
of $\sR^d(X,Y)$.

\begin{lem}$($\cite[Lemma~1.3]{ESV}$)$\label{lem:MLemma}
Assume that $k$ is infinite and let $X$ be a reduced 
quasi-projective scheme over $k$ of dimension $d \ge 2$. Then any element 
$\alpha \in \sR^d(X, X_{\rm sing})$ can be written as $\alpha = (f)_C$ for
a single reduced (but possibly reducible) Cartier curve $C$ on $X$.
\end{lem}

\begin{lem}\label{lem:sing-curve}
Let $X$ be a reduced quasi-projective scheme over $k$ of
dimension $d$ and let 
$Y \subset X$ be a closed subscheme containing $X_{\rm sing}$ and 
not containing any component of $X$. 
Let $\iota\colon X' \inj X$ be the closed immersion of a hypersurface section
of $X$ in a projective space.
Assume that $X'$ is reduced and $X' \setminus Y$ is regular.
Then there is a push-forward map $\iota_*\colon \CH^{d-1}(X', X' \cap Y)
\to \CH^d(X,Y)$.
\end{lem}
\begin{proof}
It is clear from our assumption that there is an inclusion
$\iota_*\colon \sZ^{d-1}(X',Y') \subset \sZ^d(X,Y)$, where $Y'= X' \cap Y$.
Let $C \inj X'$ be a reduced Cartier curve relative to $Y'$ and let
$(f)_C \in \sR^{d-1}(X',Y')$. Since $X' \subset X$ is a complete intersection,
it follows that $C \inj X$ is also a Cartier curve in $X$ relative to $Y$ 
and clearly $(f)_C \in \sR^d(X,Y)$. This finishes the proof.
\end{proof}

\begin{lem}\label{lem:PF}
Let $X$ be a reduced quasi-projective scheme over $k$ of
dimension $d$.
Let $\iota\colon X' \inj X$ be the closed 
immersion of a hypersurface section of $X$  in a projective space.
Assume that $X'$ is reduced and $X' \setminus X_{\rm sing}$ 
is regular.
Assume further that there is a proper map $\pi\colon X \to Z$ obtained
by a composition of blow ups at a finite collection of smooth closed 
points of a reduced quasi-projective scheme $Z$ over $k$. 
Then there is a push-forward map
$(\pi \circ \iota)_*\colon \CH^{d-1}(X', X' \cap X_{\rm sing}) \to \CH^d(Z)$.
\end{lem}
\begin{proof}
It is known that there is a push-forward map $\pi_*\colon \CH^d(X) \to
\CH^d(Z)$ which is an isomorphism (see \cite[Corollary~2.7]{ESV}).
The lemma follows by combining this with the map
$\iota_*\colon \CH^{d-1}(X', X' \cap Y) 
\to \CH^{d}(X)$ from \lemref{lem:sing-curve}.
\end{proof}

Recall from \cite[Proposition~2.1]{LW} that every regular closed point
$x \in X$ defines a class $[k(x)] \in K_0(X)$ and this yields a cycle
class map 
\begin{equation}\label{eqn:C-class}
cyc_X\colon \CH^d(X, Y) \to K_0(X). 
\end{equation}
It follows from \cite[Corollary~1.4]{ESV} that the image of this cycle class 
map does not depend on $Y$. This image is classically denoted by $F^dK_0(X)$.

For a Noetherian scheme $X$, we let $K(X)$ denote the Thomason-Trobaugh
non-connective $K$-theory spectrum of $X$. 
Recall that the connective cover of $K(X)$ coincides with the 
Quillen $K$-theory spectrum of $X$ if it is quasi-projective over $k$. 
We let $K(X; {\Z}/m)$ be the $K$-theory spectrum with coefficients ${\Z}/m$.
We let $\{K_n(X)\}_{n \in \Z}$ denote the homotopy groups (also called the
$K$-groups of $X$) of $K(X)$. From a prime number $l$,
we shall write ${\underset{m}\varinjlim} \ 
K_n(X; {\Z}/{l^m})$ by $K_n(X; {\Z}/{l^{\infty}})$.
For $n \in \Z$, we shall 
let $\sK_{n, X}$ denote the Zariski sheaf on $X$ associated
to the presheaf $U \mapsto K_n(U)$.

\subsection{Some preliminary results}\label{sec:Prelim}
For $X \in \Sch_k$, let $\sZ^F_q(X)$ denote the free abelian group of
$q$-dimensional algebraic cycles on $X$ in the sense of 
\cite[Chapter~1]{Fulton}. Let $\CH^F_{q}(X) := {\sZ^F_q(X)}/{\sR^F_q(X)}$ 
denote the associated
(homological) Chow group of $X$ modulo rational equivalence.
If $X$ is equi-dimensional of dimension $d$, we let $\CH^q_F(X)
= \CH^F_{d-q}(X)$. 
If $X$ is connected and projective over $k$, the push-forward via the
structure map $X \to \Spec(k)$ gives rise to the degree map
${\rm deg}\colon \CH^F_0(X) \to \Z$. We let ${\CH^F_0(X)}_0$ denote the kernel of
this map. In general, we define  ${\CH^F_0(X)}_0$ by taking the direct sum
over the connected components.
We prove some preliminary results in this subsection for later applications.

\begin{lem}\label{lem:Ncd}
Let $X$ be a normal quasi-projective surface over $k$ and let
$Y \subset X$ be a strict normal crossing divisor such that
$Y \cap X_{\rm sing} = \emptyset$. Then
there is a canonical isomorphism $H^2_Y(X, \sK_2) \simeq \CH^1_F(Y)$.
\end{lem}
\begin{proof}
Let $U = X \setminus X_{\rm sing}$ denote the regular locus of $X$.
Then $U$ is an open subset of $X$ which contains the closed subset $Y$.
It follows therefore from excision for the Zariski cohomology with support 
(see \cite[Chap.~III, Exercise~2.3(f)]{Hartshorne-2}) 
that the map $H^2_Y(X, \sK_2) \to H^2_Y(U, \sK_2)$ is an isomorphism. 
We can thus assume that $X$ is smooth. 

We consider the Gersten resolution
\begin{equation}\label{eqn:Ncd-0}
0 \to \sK_{2,X} \xrightarrow{\epsilon_X} i_{k(X), *}(K_2(k(X))) \to 
{\underset{z \in X^{(1)}}\oplus} i_{z,*} (K_1(k(z))) \to
{\underset{x \in X^{(2)}}\oplus} i_{x,*} (K_0(k(x))) \to 0.
\end{equation}
 
Setting $\sF = {\rm Coker}(\epsilon_X)$, we get an exact sequence
of Zariski sheaves
\[
0 \to \sF \to {\underset{z \in X^{(1)}}\oplus} i_{z,*} (K_1(k(z))) \to
{\underset{x \in X^{(2)}}\oplus} i_{x,*} (K_0(k(x))) \to 0.
\]

Since the above sequences are the flasque resolutions of
$\sK_{2,X}$ and $\sF$, respectively, it follows that the map
$H^1_Y(X, \sF) \to H^2_Y(X, \sK_{2,X})$ is an isomorphism
and there is an exact sequence
\[
H^0_Y(X, {\underset{z \in X^{(1)}}\oplus} i_{z,*} (K_1(k(z)))) \to
H^0_Y(X, {\underset{x \in X^{(2)}}\oplus} i_{x,*} (K_0(k(x)))) \to
H^1_Y(X, \sF) \to 0.
\]
Equivalently, we have an exact sequence
\begin{equation}\label{eqn:Ncd-1}
{\underset{Z}\oplus}  \ K_1(k(Z)) \xrightarrow{\rm div}  
{\underset{x \in Y^{(1)}}\oplus} K_0(k(x)) \to H^1_Y(X, \sF) \to 0,
\end{equation}
where $Z$ runs through all 1-dimensional irreducible components of $Y$.
It is well known that the first arrow from the left takes a
rational function on $Y$ to its divisor.
Hence, it is clear from the definition of $\CH^1_F(Y)$ that
this exact sequence is equivalent to an isomorphism
$\CH^1_F(Y) \xrightarrow{\simeq} H^1_Y(X, \sF)$.
Combining this with $H^1_Y(X, \sF) \xrightarrow{\simeq} H^2_Y(X, \sK_{2,X})$,
we get the desired isomorphism $\CH^1_F(Y) \xrightarrow{\simeq} 
H^2_Y(X, \sK_{2,X})$. 
\end{proof}

\begin{lem}\label{lem:tor-groups}
Let $f \colon A \to B$ be a surjective morphism of smooth connected
commutative algebraic groups over $k$. Then the induced map
$f \colon A(\ov{k}) \to B(\ov{k})$ is surjective on the 
torsion subgroups.
\end{lem}
\begin{proof}
If $A$ and $B$ are abelian varieties, then we can write ${\rm Ker}(f)$
as an extension of a finite abelian group by an abelian variety.
In particular, ${\rm Ker}(f) \otimes_{\Z} {\Q}/{\Z} = 0$.
But this easily implies that $f\colon A_{\rm tors} \to B_{\rm tors}$ is surjective.

In the general case, the structure theorem for smooth connected
commutative algebraic groups asserts that $A$ has a largest connected linear 
subgroup $A_1$ and the quotient $A_2 = A/{A_1}$ is an abelian variety. 
Moreover, $A_1$ is canonically the direct product of a torus and a smooth 
connected unipotent group (see \cite[Theorem~2.3.3]{Tits}). 
Let $B_1$ be the largest connected linear 
subgroup of $B$ such that the quotient $B_2 = B/{B_1}$ is an abelian variety. 
We have a commutative diagram of short exact sequences
\begin{equation}\label{eqn:tor-groups-00}
\xymatrix{
0 \ar[r] & A_1 \ar[r] \ar[d] & A \ar[r] \ar[d] & A_2 \ar[r] \ar[d] & 0 \\
0 \ar[r] & B_1 \ar[r] & B \ar[r] & B_2 \ar[r] & 0.
}
\end{equation}
Moreover, the corresponding maps on the linear and the abelian variety parts 
are also surjective.
We have seen above that the map $A_2 \to B_2$ is surjective on the torsion
subgroups. 

We write $A_1 = D_1 \times U_1$, $B_1 = D'_1 \times U'_1$, where 
$D_1$ (resp. $D'_1$) and $U_1$ (resp. $U'_1$) are the unique torus
and unipotent parts of $A_1$ (resp. $B_1$).
Since there are no nontrivial homomorphisms 
from a diagonalizable to a unipotent 
group and vice-versa (see \cite[\S~2.3.2]{Tits}), 
we get $ D_1 \surj D'_1$ and $U_1 \surj U'_1$.
If ${\rm char}(k) = 0$, then $U_1$ and $U'_1$ are both $k$-vector spaces by 
\cite[\S~3.1]{Tits} and hence uniquely divisible. 
If ${\rm char}(k) = p > 0$, then $U_1$ and $U'_1$ are both $p$-groups of
bounded exponents by \cite[\S~3.2.1]{Tits}. 
In particular, both groups are already torsion
and hence 
\begin{equation}\label{eqn:unip}
U_1 \otimes {\Q}/{\Z} = U'_1 \otimes {\Q}/{\Z} = 0.
\end{equation}

Next we note that the kernel of the map $D_1 \surj D_2$ is again a 
diagonalizable closed subgroup
and by \cite[Corollary~1.2.6]{Tits}, any diagonalizable group $D$ is a direct 
product of a torus
and a finite abelian group of order prime to $p$ (if ${\rm char}(k) = p > 0$).
In particular, 
\begin{equation}\label{eqn:diag}
D \otimes {\Q}/{\Z} = 0.
\end{equation}
It follows from ~\eqref{eqn:unip} and ~\eqref{eqn:diag} that 
$A_1 \to B_1$ is surjective on the torsion subgroups. They also imply that
the maps $A \to A_2$ and $B \to B_2$ are surjective on the torsion subgroups. 
A simple diagram chase in ~\eqref{eqn:tor-groups-00} now shows that the 
middle vertical map is surjective on the torsion subgroups.
\end{proof}

\begin{remk}\label{remk:Brion*} 
The assertion of \lemref{lem:tor-groups} holds even if one of 
$A$ and $B$ is not necessarily connected. This is a consequence 
of \cite[Theorem~1.1]{Brion}. But we do not need this general case in this
paper.
\end{remk}

\begin{lem}\label{lem:Blow-up-fin}
Let $X \in \Sch_k$ and let $Y \subset X$ be a proper closed subscheme.
Let $f\colon X' \to X$ be a finite morphism and let $Y' = Y \times_X X'$.
Then the induced map $\wt{f}\colon {\rm Bl}_{Y'}(X') \to {\rm Bl}_Y(X)$ is
also finite.
\end{lem}
\begin{proof}
It is enough to prove the lemma when $X = \Spec(A)$ is affine.
Let $I \subset A$ denote the ideal of definition of $Y$. Let
$X' = \Spec(B)$ so that there is a finite ring homomorphism
$A \to B$. Set $J = IB$. 

We can find a polynomial ring $A[x_0, \cdots , x_n]$ and a 
surjective graded ring homomorphism $A[x_0, \cdots , x_n] \surj A[It]$,
where $A[It]$ is the Rees-algebra (recalled in the
proof of \propref{prop:pro-desc-CM-A}) of $I$ over $A$ .
This yields a surjection $B[x_0, \cdots , x_n] \surj B[Jt]$ and 
hence a commutative diagram
\[
\xymatrix@C1pc{
{\rm Bl}_{Y'}(X') \ar[r] \ar[d] & \P^n_{X'} \ar[d] \\
{\rm Bl}_Y(X) \ar[r] & \P^n_X.}
\]

It is clear that the horizontal arrows are closed immersions.
Since the right vertical arrow is clearly finite, it follows
that the left vertical arrow is also finite.
\end{proof}

\vskip .3cm

\section{Torsion in bi-relative $K$-theory}\label{sec:pro-desc}
Let $k$ be any field.
In this section, we use some results of \cite{GH-1} to study the
torsion in the bi-relative $K$-groups associated to certain abstract blow-ups
in $\Sch_k$. The results of this section will be used in the
proof of the torsion theorem for surfaces. In \S~\ref{sec:Pf:Chow-res},
we shall apply these results to prove \thmref{thm:Chowgroup}.

\subsection{Pro-spectra and their weak equivalence}
\label{sec:pro-desc-**}
Recall that a pro-object in a category $\sC$ consists of 
a covariant functor from a small cofiltering category to $\sC$. 
A good exposition of pro-objects in a category can be found in 
\cite[\S~2]{Isaksen}.
In this paper, a pro-object in a category $\sC$ will always mean a sequence 
$\{A_1 \xleftarrow{\alpha_1} A_2 \xleftarrow{\alpha_2} \cdots \}$ of objects 
in $\sC$. It will be formally denoted by $``{{\underset{i}\varprojlim}}" A_i$. 
A morphism $f\colon ``{{\underset{i}\varprojlim}}" A_i \to
``{{\underset{j}\varprojlim}}" B_j$ in the category ${\rm pro}\sC$ of
pro-objects in $\sC$ is an element of the set
${\underset{j}\varprojlim} {\underset{i}\varinjlim} \
\Hom_{\sC}(A_i, B_j)$. 
In particular, such a morphism $f$ is the
same as giving a function $\lambda\colon \N^{+} \to \N^{+}$ and a morphism
$f_i \colon A_{\lambda(i)} \to B_{i}$ in $\sC$ for each $i \ge 1$
such that for any $j \ge i$, there is some $l \ge \lambda(i),  \lambda(j)$
so that the diagram
\begin{equation}\label{eqn:Ind-Obj}
\xymatrix@C1pc{
A_l \ar[r] \ar[dr] & A_{\lambda(j)} \ar[r]^-{f_j} &  B_j \ar[d] \\
& A_{\lambda(i)} \ar[r]_-{f_i} & B_i}
\end{equation}
commutes in $\sC$.
We shall call such a morphism to be {\sl strict} if
$\lambda$ is the identity function.

If $\sC$ admits cofiltered limits,
the limit of $``{{\underset{i}\varprojlim}}" A_i$ will be denoted by
$\underset{i}\varprojlim \ A_i$.   
If $\sC$ is an abelian category, then so is ${\rm pro}\sC$.
If $f \colon ``{{\underset{i}\varprojlim}}" A_i \to
``{{\underset{i}\varprojlim}}" B_i$ is a strict morphism, then one checks
easily that ${\rm Ker}(f) = ``{{\underset{i}\varprojlim}}" {\rm Ker}(f_i)$
and ${\rm Coker}(f) = ``{{\underset{i}\varprojlim}}" {\rm Coker}(f_i)$.
In particular, a sequence of strict morphisms of pro-objects 
\begin{equation}\label{eqn:pro-exact}
``{{\underset{i}\varprojlim}}" A_i \to
``{{\underset{i}\varprojlim}}" B_i \to
``{{\underset{i}\varprojlim}}" C_i 
\end{equation}
is exact in the abelian category ${\rm pro}\sC$ if it restricts to an exact 
sequence of objects in $\sC$ for each $i \in \N^{+}$.
We should warn that the exactness of ~\eqref{eqn:pro-exact} does not 
imply that the sequence 
remains exact if we replace $``{{\underset{i}\varprojlim}}"$ by 
$\underset{i}\varprojlim$. 
We refer the reader to \cite[Appendix~4]{AM} for these facts about pro-objects
in abelian categories.

Consider a Cartesian square in $\Sch_k$:
\begin{equation}\label{eqn:abs-blow-up-0}
\xymatrix@C1pc{
Y' \ar[r]^{\iota'} \ar[d]_{g} & X' \ar[d]^{f} \\
Y \ar[r]_{\iota} & X.}
\end{equation}

Recall that ~\eqref{eqn:abs-blow-up-0} is called an
{\sl abstract blow-up square} 
if $\iota$ is a closed immersion, $f$ is proper and the induced map
$X' \setminus Y' \to X \setminus Y$ is an isomorphism.
We say that ~\eqref{eqn:abs-blow-up-0} is a {\sl finite abstract 
blow-up square} if it is an abstract blow-up square such that $f$ is finite. 

Given a presheaf of (possibly non-connective) 
$S^1$-spectra $\sL$ on $\Sch_k$, we let $\sL(X, X')$ denote the homotopy
fiber of the map of spectra $g^*\colon \sL(X) \to \sL(X')$.
We let $\sL(X,X', Y, Y')$ denote the homotopy fiber of the map of
spectra $\sL(X, Y) \to \sL(X', Y')$. It is easy to check that
$\sL(X,X', Y, Y')$ is same as the homotopy fiber of the map of
spectra $\sL(X, X') \to \sL(Y, Y')$ in the homotopy category of spectra. 
We shall often denote the homotopy
groups $\pi_n(\sL(X,X', Y, Y'))$ by $\sL_n(X,X', Y, Y')$ for $n \in \Z$.

We shall say that a presheaf of spectra $\sL$ on $\Sch_k$ satisfies 
pro-descent for the square ~\eqref{eqn:abs-blow-up-0} if
the induced map of pro-abelian groups
$``{{\underset{r}\varprojlim}}" \pi_n(\sL(X, rY))$
$\to ``{{\underset{r}\varprojlim}}" \pi_n(\sL(X', rY'))$ is an
isomorphism for all $n \in \Z$. Note that a pro-descent of 
presheaves of spectra can also be defined using suitable model structures on
the presheaves of pro-spectra on $\Sch_k$ (see, for instance, \cite{Isak-1}).
But we do not use descent in that generality in this paper. 

\subsection{Relative and bi-relative $K$-theory}
\label{sec:boundedness}
If we let $\sL$ denote the presheaf of $K$-theory spectra 
on $\Sch_k$ (see the last part of \S~\ref{sec:Zero-C}),
we shall call $K_n(X, Y)$ and $K_n(X,X', Y, Y')$ the relative
and the bi-relative $K$-groups, respectively.


Let $f\colon X' \to X$ be a finite morphism in $\Sch_k$. Let $\sI_Y$ be a sheaf
of ideals on $X$ such that the map $\sI_Y \to 
f_*(f^{-1}\sI_Y \cdot \sO_{X'}) = \sI_Yf_*(\sO_{X'})$ is an
isomorphism. Let $\iota\colon Y \inj X$ be the closed subscheme defined by 
$\sI_Y$ and let $Y' = Y \times_X X'$. In this case, we shall say that 
$Y$ is a conducting subscheme and 
write $K_n(X,X', Y, Y')$ in short as $K_n(X, X', Y)$.
For any locally closed subscheme $U \subset X$, we let $Y_U = Y \times_X U$
and $U' = X' \times_X U$.

\begin{lem}\label{lem:MV-bi-rel}
Let $U, V \subset X$ be two open subsets such that $X = U \cup V$
and let $W = U \cap V$. Then the induced square of spectra
\begin{equation}\label{eqn:MV-bi-rel-00}
\xymatrix@C1pc{
K(X,X',Y) \ar[r] \ar[d] & K(U, U', Y_U) \ar[d] \\
K(V, V', Y_V) \ar[r] & K(W, W', Y_W)}
\end{equation}
is homotopy Cartesian.
\end{lem}
\begin{proof}
Let $K^{X \setminus U}(X)$ denote the homotopy fiber of the restriction map
$K(X) \to K(U)$. Then it follows from \cite[Theorem~8.1]{TT} that the
map $K^{X \setminus U}(X) \to K^{V \setminus W}(V)$ is a homotopy equivalence
and the same holds if we replace $X$ by $Y$. The homotopy fiber sequence
\[
K^{X \setminus U}(X, Y) \to K^{X \setminus U}(X) \to K^{Y \setminus U}(Y)
\]
shows that the map
$K^{X \setminus U}(X, Y) \to K^{V \setminus W}(V, Y_V)$ is a homotopy equivalence.
The same is also true if we replace $X$ by $X'$. 

Notice now that $K^{X \setminus U}(X, Y)$ is also the homotopy fiber of the
map $K(X,Y) \to K(U, Y_U)$. 
In particular, we have a commutative diagram of spectra
\begin{equation}\label{eqn:MV-bi-rel-1}
\xymatrix@C1pc{
K^{X \setminus U}(X, X',Y) \ar[r] \ar[d] & K(X, X', Y) \ar[r] \ar[d] &
K(U, U', Y_U) \ar[d] \\
K^{X \setminus U}(X, Y) \ar[r] \ar[d] & K(X,Y) \ar[r] \ar[d] & K(U, Y_U) \ar[d] 
\\
K^{X' \setminus U'}(X', Y') \ar[r] & K(X',Y') \ar[r] & K(U', Y'_{U'}),}
\end{equation}
where all rows and columns are homotopy fiber sequences, and these 
fiber sequences uniquely define $K^{X \setminus U}(X, X',Y)$.
The commutative diagram of homotopy fiber sequences
\begin{equation}\label{eqn:MV-bi-rel-2}
\xymatrix@C1pc{ 
K^{X \setminus U}(X, X',Y) \ar[r] \ar[d] & 
K^{X \setminus U}(X, Y) \ar[r] \ar[d] &
K^{X' \setminus U'}(X', Y') \ar[d] \\
K^{V \setminus W}(V, V',Y_V) \ar[r] & 
K^{V \setminus W}(V, Y_V) \ar[r] &
K^{V' \setminus W'}(V', Y'_{V'})}
\end{equation}
now shows that the map $K^{X \setminus U}(X, X',Y) \to K^{V \setminus W}(V, V',Y_V)$
is a homotopy equivalence. But this is equivalent to saying that
~\eqref{eqn:MV-bi-rel-00} is homotopy Cartesian.
\end{proof}

An argument identical to the proof of \lemref{lem:MV-bi-rel} also proves
the following.

\begin{lem}\label{lem:MV-rel}
Let $U, V \subset X$ be two open subsets such that $X = U \cup V$
and let $W = U \cap V$. Let $Y \subset X$ be a closed subscheme.
Then the induced square of spectra
\begin{equation}\label{eqn:MV-bi-rel-0}
\xymatrix@C1pc{
K(X,Y) \ar[r] \ar[d] & K(U, Y_U) \ar[d] \\
K(V, Y_V) \ar[r] & K(W, Y_W)}
\end{equation}
is homotopy Cartesian.
\end{lem}

\subsection{Torsion in relative and bi-relative $K$-groups}
\label{sec:Tor-rel}
We let $k$ be a field of characteristic $p > 0$ and keep the notations of
\S~\ref{sec:boundedness}. Our goal in this subsection
is to show using the Zariski descent results of
\S~\ref{sec:boundedness} that various relative and bi-relative $K$-groups
associated to a resolution of singularities of a singular scheme
are torsion groups of bounded exponents. We begin with the following.

\begin{lem}\label{lem:p-fin-exp}
Suppose that ~\eqref{eqn:abs-blow-up-0} is a conductor square
and let $n \in \Z$ be any integer. 
Then the bi-relative $K$-group $K_n(X,X',Y)$ is a 
$p$-primary torsion group of bounded exponent.
\end{lem}
\begin{proof} 
If $X$ is affine, this follows from \cite[Theorem~C]{GH-1}.
In general, we shall argue by induction on the minimal number of affine open
subsets that cover $X$.
Let us write $X = U_1 \cup \cdots \cup U_r$, where each $U_i$ is affine
open in $X$. Set $U = U_1, V = U_2 \cup \cdots \cup U_r$ and
$W = U \cap V = (U \cap U_2) \cup \cdots \cup (U \cap U_r)$.

It follows from \lemref{lem:MV-bi-rel} that for every $n \in \Z$,
there is an exact sequence
\[
K_{n+1}(W, W', Y_W) \to K_n(X, X', Y) \to K_n(U, U', Y_U) \oplus
K_n(V, V', Y_V).
\]

Since our schemes are all separated, each $U \cap U_i$ is affine for
$2 \le i \le r$. Since $U$ is affine, $K_n(U, U', Y_U)$ is a $p$-primary
torsion group of bounded exponent. It follows by induction on the minimal 
number of open subsets in an affine open cover that
$K_{n+1}(W, W', Y_W)$ and $K_n(V, V', Y_V)$ are $p$-primary
torsion groups of bounded exponents. We deduce easily that 
$K_n(X, X', Y)$ is also a $p$-primary group torsion of bounded exponent.
\end{proof}

Using \cite[Theorem~A]{GH-1} and \lemref{lem:MV-rel},
an argument identical to the proof
of \lemref{lem:p-fin-exp} proves the following result.

\begin{lem}\label{lem:p-fin-exp-nil}
Let $k$ be a field of characteristic $p > 0$. 
Let $X \in \Sch_k$ and let $Y \subset X$ be a closed subscheme whose sheaf of 
ideals in $\sO_X$ is nilpotent. Then for every integer $n \in \Z$, the 
relative $K$-group $K_n(X,Y)$ is a $p$-primary torsion group of bounded 
exponent.
\end{lem}

\begin{lem}\label{lem:fini-exp-nil}
Let $k$ be a field of characteristic $p > 0$ and let $X \in \Sch_k$. 
Let $Z, Y \subset X$ be closed subschemes 
such that $Z_{\rm red} = Y_{\rm red}$. In the commutative square
~\eqref{eqn:abs-blow-up-0}, let $Z' = Z \times_X X'$. 
Then for any $n \in \Z$, the group $K_n(X, X', Y, Y')$ is
$p$-primary torsion of bounded exponent if and only if
so is $K_n(X, X', Z, Z')$.
\end{lem}
\begin{proof}
It follows from our assumption that $Z \subset mY \subset nZ$ for 
all $n \gg m \gg 0$. It suffices therefore to prove the lemma when
$Z \subset Y$.

In this case, we have a commutative diagram of spectra  
\begin{equation}\label{eqn:fini-exp-fin-0}
\xymatrix@C.8pc{
K(X, X', Y, Y') \ar[r] \ar[d] & 
K(X,X', Z, Z') \ar[r] \ar[d] & K(Y,Y', Z, Z') \ar[d] \\
K(X, Y) \ar[r] \ar[d] &  K(X, Z) \ar[r] \ar[d] &  K(Y, Z) \ar[d] \\ 
K(X', Y') \ar[r] &  K(X', Z') \ar[r] &  K(Y', Z')}
\end{equation}
where the rows and columns are homotopy fiber sequences. 
In particular, there is an exact sequence of bi-relative $K$-groups
\begin{equation}\label{eqn:fini-exp-fin-1}
K_{n+1}(Y, Y', Z, Z') \to K_n(X, X', Y, Y') \to K_n(X, X', Z, Z') \to
K_{n}(Y, Y', Z, Z'). 
\end{equation}

It suffices therefore to show that 
for every $n \in \Z$, the bi-relative $K$-group
$K_{n}(Y, Y', Z, Z')$ is a $p$-primary torsion group of bounded exponent.
But this is an immediate consequence of the vertical fiber
sequence on the right end of ~\eqref{eqn:fini-exp-fin-0}
and \lemref{lem:p-fin-exp-nil}.
\end{proof}

\begin{lem}\label{lem:conductor}
Given a finite abstract blow-up square ~\eqref{eqn:abs-blow-up-0}, there
exists a closed subscheme $Z \subset X$ such that $Z_{\rm red} = Y_{\rm red}$
and $\sI_Z = f_* \circ f^*(\sI_Z)$.
\end{lem}
\begin{proof}
We let $\sI$ denote the sheaf of ideals on $X$ defining $Y$.
Let $X'' \subset X$ denote the scheme-theoretic image of the finite map $f$
so that locally one has $X'' = \Spec({\sO_X}/{\sJ})$, where $\sJ= 
{\rm Ker}(\sO_X \to f_*(\sO_{X'}))$. There is a factorization
$X' \xrightarrow{f'} X'' \xrightarrow{f''} X$ so that  $f'$ is 
finite and surjective, $f''$ is a closed immersion and $f = f'' \circ f'$.
We have exact sequences of coherent $\sO_X$-modules   
\begin{equation}\label{eqn:conductor-0}
0 \to \sJ \to \sO_X \to \sO_{X''} \to 0;  \ \ \
0 \to \sO_{X''} \to f_*(\sO_{X'}) \to 
{f_*(\sO_{X'})}/{\sO_{X''}} \to 0.
\end{equation}

Since the maps $X' \to X'' \to X$ are isomorphisms away from $Y$, one easily
checks that $\sI^n \sJ = 0$ and $\sI^nf_*(\sO_{X'}) \subset \sO_{X''}$
for all $n \gg 0$. It follows using the Artin-Rees lemma in commutative
algebra (see \cite[Theorem~8.5]{Matsumura}) that for all $n \gg 0$,
one has $\sI^n \cap \sJ = 0$ so that $\sI^n = \sI^nf''_*(\sO_{X''})$. 
We let $n_0 \ge 1$ be the smallest integer such that this happens. 
If we let $\sI_1 = \sI^{n_0}$ and $\wt{\sI} = \sI_1 f'_*(\sO_{X'})$, then
we have an inclusion of the sheaves of ideals $\sI_1 \subseteq \wt{\sI}$
and $\wt{\sI}$ has the property that $\wt{\sI} = \wt{\sI} f'_*(\sO_{X'}) =
\wt{\sI} f_*(\sO_{X'})$. The lemma is therefore reduced to showing that
$\sI_1$ and $\wt{\sI}$ have the same support in $X''$.

To show this, we let $Y'' = Y \times_X X''$ (so that $Y' = Y'' \times_{X''} X'$)
and let $Z \subset X''$
be the closed subscheme defined by $\wt{\sI}$. We then observe that
the map $f'\colon n_0Y' \to n_0Y''$ factors through $n_0Y' \to Z \inj n_0Y''$.
On the other hand, $f'$ is finite and surjective and hence the composite
map $n_0Y' \to Z \inj n_0Y''$ is also finite and surjective. 
But this implies that
$Z_{\rm red} = Y''_{\rm red}$. This finishes the proof.
\end{proof}

\begin{prop}\label{prop:fini-exp-fin}
Let $k$ be a field of characteristic $p > 0$.
Given a finite abstract blow-up square ~\eqref{eqn:abs-blow-up-0}
and an integer $n \in \Z$, 
the bi-relative $K$-group $K_n(X, X', Y, Y')$ is a $p$-primary torsion
group of bounded exponent. 
\end{prop}
\begin{proof}
By \lemref{lem:conductor}, we can find a closed subscheme 
$Z \subset X$ whose support is same as that of $Y$ such that 
$\sI_Z = \sI_Z f_*(\sO_{X'})$. We let $Z' = Z \times_X X'$.
By \lemref{lem:fini-exp-nil}, it is enough to show that
$K_n(X, X', Z, Z')$ is a $p$-primary torsion
group of bounded exponent. But this follows from 
\lemref{lem:p-fin-exp}.
\end{proof}

\begin{prop}\label{prop:pro-desc-CM-A}
Let $k$ be a field of characteristic $p > 0$.
Let $X \in \Sch_k$ be a reduced scheme with only isolated Cohen-Macaulay
singularities. Let $Y$ denote the singular locus
of $X$ with reduced induced closed subscheme structure.
Let $Z \subset X$ be a closed subscheme such that $Z_{\rm red} = Y$. 
Let $\pi\colon \wt{X} \to X$ denote the blow-up of $X$ along $Z$ and let
$E \subset \wt{X}$ denote the reduced exceptional divisor.
Then for every $n \in \Z$ and $r \ge 1$, the bi-relative $K$-group
$K_n(X, \wt{X}, rY, rE)$ is a $p$-primary torsion group of bounded
exponent.
\end{prop}
\begin{proof}
Letting $F_r$ denote the homotopy fiber of the map
$K(rY, Y) \to K(rE, E)$, there is a homotopy fibration sequence
$K(X, \wt{X}, rY, rE) \to K(X, \wt{X}, Y, E) \to F_r$.
Using \lemref{lem:p-fin-exp-nil}, we need only prove the proposition when
$r =1$.

Recall that for an ideal $I$ in a commutative ring $R$,
the Rees-algebra (also called the Blow-up algebra) $R(I)$ is the
graded $R$-algebra $\oplus_{n \ge 0} I^n$, where the summand $I^n$
is placed in degree $n$. The graded multiplication of $R(I)$ is
induced by the usual multiplication of various powers of $I$ in $R$.
Recall that a Noetherian scheme is called Cohen-Macaulay if each of its local
rings is a Cohen-Macaulay ring. 

Let $\sI_Z$ be the sheaf of ideals
on $X$ defining $Z$. By the Northcott-Rees theory of reduction of ideals,
there exists $m \ge 1$ and a minimal reduction ideal $\sJ$
of $\sI^m_Z$, where one can take $m =1$ if $k$ is infinite
(see \cite[Theorem~14.14]{Matsumura}). 
Here, we say that an ideal $\sJ \subset \sI_Z$ is a minimal
reduction of $\sI_Z$ if the stalks of $\sJ$ along $Y$ are generated by
$\dim(X)$ many elements and
the map ${\rm Proj}_X(\sR(\sI_Z)) \to {\rm Proj}_X(\sR(\sJ))$ is 
a finite morphism, where $\sR(\sI_Z)$ denotes the sheaf of Rees-algebras
of $\sI_Z$ over $\sO_X$ (see \cite[Theorem~1.5]{Weibel-1}).
Since the blow-up of $\sI_Z$ is unchanged if we replace $\sI_Z$ by
its powers, we can assume $m =1$ to obtain a reduction ideal.

Now, since $X$ is Cohen-Macaulay with only isolated singularities,
we can further assume that the reduction ideal $\sJ$ is a local 
complete intersection ideal sheaf in $\sO_X$ 
(see \cite[Proposition~1.6]{Weibel-1}). Setting
$X' = {\rm Proj}_X(\sR(\sJ))$, this gives rise to a 
commutative diagram
\begin{equation}\label{eqn:Desc-CM-0}
\xymatrix@C1pc{
\wt{X} \ar[r]^{f} \ar@/_1pc/[rr]_{\pi} & X' \ar[r]^{\pi'} & X,}
\end{equation}
where $f$ is finite and $\pi'$ is the blow-up along a regular closed 
immersion $W \subset X$ such that $Z \subset W$ with $W_{\rm red} = Y$.
By \lemref{lem:fini-exp-nil}, it suffices to show that
$K_n(X, \wt{X}, W, \wt{W})$ is a $p$-primary torsion group of bounded
exponent, where  $\wt{W} = W \times_X \wt{X}$.

We set $W' =  W \times_X X'$ so that  $\wt{W} = W' \times_{X'} \wt{X}$. 
It follows from \cite[Th{\'e}or{\`e}me~2.1]{Thomason-1} that 
the map  $K(X, W) \to K(X', W')$ is a homotopy equivalence.
On the other hand, we have a commutative diagram
\[
\xymatrix@C.8pc{
K(X, X', W, W') \ar[r] \ar[d] & K(X, W) \ar[r] \ar[d] & K(X', W') \ar[d] \\
K(X, \wt{X}, W, \wt{W}) \ar[r] & K(X, W) \ar[r] & K(\wt{X}, \wt{W})}
\]  
where the two rows are homotopy fiber sequences.
In particular, we get a homotopy fiber sequence
$K(X, X', W, W') \to K(X, \wt{X}, W, \wt{W}) \to K(X', \wt{X}, W', \wt{W})$.

We have shown above that $K(X, X', W, W')$ is contractible
and it follows from \propref{prop:fini-exp-fin} that
$K_n(X', \wt{X}, W', \wt{W})$ is $p$-primary torsion of bounded exponent.
We conclude that the same holds for $K_n(X, \wt{X}, W, \wt{W})$ too. 
This finishes the proof.
\end{proof} 

Let $X \in \Sch_k$ be singular.
Recall that a proper morphism of schemes $\pi\colon \wt{X} \to X$ is called 
a resolution of singularities, if its restriction
to the regular locus of $X$ is an isomorphism and $\wt{X}$ is regular.
We say that $\pi$ is a {\sl good resolution of singularities},
if it is obtained as a blow-up of $X$ along a closed subscheme
whose support is the singular locus of $X$.
It is well-known that a good resolution of singularities always exists in
characteristic zero. In characteristic $p > 0$, it exists
if either $\dim(X) \le 2$ or $\dim(X) = 3$ and $k$ is perfect
(see \cite{VO-1}, \cite{VO-2}).

\begin{thm}\label{thm:Desc-CM-Main-A}
Let $k$ be a field of characteristic $p > 0$. Let $X$ be a reduced 
quasi-projective $k$-scheme with only isolated Cohen-Macaulay singularities. 
Let $Y$ denote 
the singular locus of $X$ with reduced induced closed subscheme structure.
Let $\pi\colon\wt{X} \to X$ be a good resolution of singularities of $X$ 
with reduced exceptional divisor $E$. 
Then for every $n \in \Z$ and $r \ge 1$, the bi-relative $K$-group
$K_n(X, \wt{X}, rY, rE)$ is a $p$-primary torsion group of bounded
exponent.
\end{thm}
\begin{proof}
By our assumption, $\pi\colon \wt{X} \to X$ is the blow-up of $X$ along a closed 
subscheme $Z \subset X$ with $Z_{\rm red} = Y$.
We can therefore apply \propref{prop:pro-desc-CM-A} to conclude the proof.
\end{proof}

\vskip .3cm

\section{Divisibility of $SK_1$ of normal 
projective surfaces}\label{sec:nor-proj}
Recall that for any Noetherian scheme $X$, there is a natural map
$K_1(X) \to H^0(X, \sK_{1, X})$ and $SK_1(X)$ is defined to be the
kernel of this map.
Our goal in this section is to prove the divisibility property
of the group $SK_1(X)$ for a normal projective 
surface $X$.

Let $Z$ be a Noetherian scheme.
For $n \in \Z$ and closed subscheme $Z' \subset Z$, 
let $\sK_{n, (Z, Z')}$ denote the Zariski sheaf on $Z$ associated
to the presheaf $U \mapsto K_n(U, U \cap Z')$ (see end of 
\S~\ref{sec:Zero-C}). The Zariski sheaves of 
bi-relative $K$-theory $\sK_{n, (Z, W, Z')}$ on $Z$ are defined similarly.
Given a closed subscheme $W \subset Z$, the pull-back map 
$K_1(Z, W) \to K_1(U, W \cap U)$ for any open subset 
$U \subset Z$ defines a natural map 
${\rm det}_{Z|W}\colon K_1(Z, W) \to H^0(Z, \sK_{1, (Z, W)})$. 
We let $SK_1(Z,W) := {\rm Ker}({\rm det}_{Z|W})$ so that 
$SK_1(Z) = SK_1(Z, \emptyset)$.

The following result is an easy consequence of the Thomason-Trobaugh
spectral sequence.

\begin{lem}\label{lem:SK-1-exact}
Let $Z$ be a Noetherian scheme of dimension at most two and let
$W \subset Z$ be a closed subscheme. Then there is a natural
exact sequence
\begin{equation}\label{eqn:PNormal-0-ex}
H^2(Z, \sK_{3, (Z,W)}) \to SK_1(Z,W) \to H^1(Z, \sK_{2, (Z,W)}) \to 0.
\end{equation}
\end{lem}
\begin{proof}
Apply the Thomason-Trobaugh spectral sequence
$E^{p,q}_2 = H^p(X, \sK_{q, (Z,W)}) \Rightarrow K_{q-p}(Z,W)$
having differentials $d_r\colon E^{p,q}_r \to E^{p+r, q+r-1}_r$
and use the fact that the Zariski cohomological dimension of $Z$
is at most two.
\end{proof}

\subsection{$SK_1$ of normal projective surfaces}\label{sec:SK-Nor}
Let $k$ be an algebraically closed field of characteristic $p > 0$. 
This assumption will be used throughout the rest of \S~\ref{sec:nor-proj}.
Let $X$ be an irreducible normal projective surface over $k$
and let $Y \subset X$ denote its singular locus with the reduced induced
subscheme structure. We assume $Y \neq \emptyset$.
Using the resolution of singularities for surfaces,
we can find a resolution of singularities $\pi\colon \wt{X} \to X$ of
$X$ such that the reduced exceptional divisor $E \subset \wt{X}$ is 
a strict normal crossing divisor.

Since the map of Zariski sheaves $\sK_{2,X} \to \sK_{2,rY}$ is surjective
(both being same as the sheaves of corresponding Milnor $K_2$-groups),
there is an injection $\sK_{1, (X,rY)} \inj  \sK_{1, X}$.
Since $X$ is irreducible and projective over $k = \ov{k}$,
we have $H^0(X, \sK_{1, X}) \simeq k^{\times}$. It follows that
$H^0(X, \sK_{1, (X,rY)}) = 0$ (here we use $Y \neq \emptyset$). 
Similarly, we have 
$H^0(\wt{X}, \sK_{1, (\wt{X},rE)}) = 0$. 
It follows that 
\begin{equation}\label{eqn:SK-1-K1}
SK_1(X, rY) \xrightarrow{\simeq} K_1(X,rY)\ \ \mbox{and} \ \
SK_1(\wt{X}, rE) \xrightarrow{\simeq} K_1(\wt{X},rE) \ \ \forall \ r \ge 1.
\end{equation}

\begin{remk}
We remark that there is no need to assume $Y \neq \emptyset$
to prove the remaining results of this section, provided we know 
in general for a closed immersion $W \subset Z$ of
Noetherian schemes that there is a functorial decomposition 
$K_1(Z,W) \simeq SK_1(Z,W) \oplus H^0(Z, \sK_{1, (Z,W)})$.
But this is indeed known to be true and we give a sketch of its proof.
\end{remk}

\begin{lem}\label{lem:Rel-surj-K1}
The map ${\rm det}_{Z|W}\colon K_1(Z, W) \to H^0(Z, \sK_{1, (Z, W)})$ is split 
surjective
and the splitting is natural in the pair $(Z,W)$.
\end{lem}
\begin{proof}
We shall define a natural homomorphism
$\beta_{Z|W}\colon H^0(Z, \sK_{1, (Z, W)}) \to K_1(Z, W)$
such that ${\rm det}_{Z|W} \circ \beta_{Z|W}$ is identity.

We let $\G_m$ be the group scheme $\Spec(\Z[t^{\pm 1}])$ over $\Spec(\Z)$.
An element $f \in H^0(Z, \sK_{1, (Z, W)})$ is equivalent to a morphism
$f\colon Z \to \G_{m} \subset \A^1_{\Z}$ such that $f(W) \subset \{t = 1\}$. 
This gives rise to a commutative diagram with exact rows
\begin{equation}\label{eqn:Rel-surj-K1-0}
\xymatrix@C.8pc{
0 \ar[r] & K_1(\G_m, \{t = 1\}) \ar[r] \ar[d]_{f^*} & K_1(\G_m) \ar[d]^{f^*} 
\ar[r] & K_1(\{t = 1\}) \ar[d]^{f^*} \ar[r] & 0 \\
& K_1(Z,W) \ar[d] \ar[r] & K_1(Z) \ar[d]^{\delta} \ar[r] & K_1(W) \ar[d] & \\
0 \ar[r] & H^0(Z, \sK_{1, (Z, W)}) \ar[r] & H^0(Z, \sK_{1, Z}) \ar[r] &
H^0(W, \sK_{1, W}).}
\end{equation}

One can now check (as is well known)
that $\delta \circ f^*([t]) = f$. Since $[t] \in K_1(\G_m, \{1\})$, we see 
that $f^*([t]) \in K_1(Z,W)$ and hence $\delta \circ f^*([t])$ dies in
$H^0(W, \sK_{1, W})$. It must therefore lie in $H^0(Z, \sK_{1, (Z, W)})$.
Letting $\beta_{Z|W}(f) = f^*([t]) \in K_1(Z,W)$, we get the desired
splitting of ${\rm det}_{Z|W}$. 
It is clear from the above construction that $\beta_{Z|W}$ is natural in 
$(Z,W)$.

To show the additivity of $\beta_{Z|W}$, let $f, g \in  
H^0(Z, \sK_{1, (Z,W)})$ and consider the maps
\begin{equation}\label{eqn:multi}
\xymatrix@C.8pc{
& \G_m & \\
X \ar[r]^/-1pc/{h} \ar[dr]_{g} \ar[ur]^{f} & \G_m \times \G_m \ar[r]^/.5pc/{\mu} 
\ar[d]^{p_2} \ar[u]_{p_1} & \G_m \\
& \G_m, & }
\end{equation}
where $p_1, \ p_2$ are the projections, $h(x) = (f(x), g(x))$ and
$\mu$ is the multiplication map. All these maps take $W$ to $\{t = 1\} \subset 
\G_m$ or to $(\{t = 1\} \times \{t = 1\}) \subset \G_m \times \G_m$.

Since $\mu$ is induced by the map on functions
$\mu_*\colon \Z[t^{\pm 1}] \to \Z[x^{\pm 1}, y^{\pm 1}]$, given by
$\mu_*(t) = xy$, we get
\[
\begin{array}{lllll}
\beta_{Z|W}(fg) & = & (fg)^*([t]) & = & (\mu \circ h)^*([t]) \\
& = & h^* \circ \mu^*([t]) & = & h^*([x] + [y]) \\
& = & (p_1 \circ h)^*([t]) & + & (p_2 \circ h)^*([t]) \\
& = & f^*([t]) + g^*([t]) & = & \beta_{Z|W}(f) + \beta_{Z|W}(g).
\end{array}
\]
It follows that $\beta_{Z|W}$ is a group homomorphism. 
\end{proof}

\vskip .3cm

We now return to our original pair $(X,Y)$ and prove:

\begin{lem}\label{lem:PNormal-SK1-A}
For every $r \ge 1$, the map
\[
{SK_1(X,rY)}/{p^{\infty}}  \to {SK_1(X)}/{p^{\infty}} 
\]
is an isomorphism.
\end{lem}
\begin{proof}
We first claim that the map $SK_1(X,rY) \to SK_1(X)$ is surjective.
Using \lemref{lem:SK-1-exact}, it suffices to show that the map
$H^i(X, \sK_{i+1, (X, rY)}) \to  H^i(X, \sK_{i+1, X})$ is surjective
for $i = 1,2$.
We have an exact sequence of Zariski sheaves
\[
\sK_{n+1, rY} \to \sK_{n, (X, rY)} \to \sK_{n,X} \to \sK_{n, rY}.
\]
In particular, the kernel and the cokernel of the middle arrow
are supported on $Y$. The desired surjectivity for the induced maps
on the cohomology groups follows easily (see for instance, 
\cite[Lemma~1.3]{PW}).
This proves the claim. 

Since $H^0(X, \sK_{1, (X, rY)}) =0$, the long exact sequence for the relative 
$K$-theory of the pair $(X, rY)$
now gives us an exact sequence
\begin{equation}\label{eqn:PNormal-SK1-0}
K_2(rY) \to SK_1(X, rY) \to  SK_1(X) \to 0.
\end{equation} 

Next, we note that we can write $rY$ as a scheme-theoretically finite
sum $rY = \amalg_i rY_i$, where the structure map $Y_i \to \Spec(k)$
is an isomorphism for every $i$.
In particular, the inclusion $Y \subset rY$ has a section
and this yields a decomposition $K_2(rY) = K_2(rY, Y) \oplus K_2(Y)$.
It follows from \lemref{lem:p-fin-exp-nil} that $K_2(rY, Y)$ is a $p$-primary
torsion group of bounded exponent. Since ${\Q_p}/{\Z_p}$ is divisible,
we must have ${K_2(rY, Y)}/{p^{\infty}} = 0$.
On the other hand, $K_2(Y)$ is divisible and hence
${K_2(Y)}/{p^{\infty}} \simeq
{\underset{m}\varinjlim}\ {K_2(Y)}/{p^m} = 0$.
All of this uses that $k$ is algebraically closed.
It follows that the map ${SK_1(X, rY)/{p^{\infty}}} \to
{SK_1(X)}/{p^{\infty}}$ is an isomorphism.
\end{proof}

\begin{lem}\label{lem:SK1-div-smooth}
${SK_1(\wt{X})}/{p^{\infty}} = 0$.
\end{lem}
\begin{proof}
Using the exact sequence
\begin{equation}\label{eqn:SK1-div-smooth-0}
H^2(\wt{X}, \sK_{3, \wt{X}}) \to SK_1(\wt{X}) \to H^1(\wt{X}, \sK_{2, \wt{X}})
\to 0,
\end{equation}
it suffices to show that ${H^2(\wt{X}, \sK_{3, \wt{X}})}/{p^{\infty}} = 0 =
{H^1(\wt{X}, \sK_{2, \wt{X}})}/{p^{\infty}}$.

The Gersten resolution for $\sK_{3, \wt{X}}$ tells us that 
$H^2(\wt{X}, \sK_{3, \wt{X}})$ is the quotient of the group
${\underset{x \in {\wt{X}}^{(2)}} \oplus} \ K_1(k(x))$, which is clearly
divisible (since $k = \ov{k}$). This yields 
${H^2(\wt{X}, \sK_{3, \wt{X}})}/{p^{\infty}} = 0$. The vanishing of
${H^1(\wt{X}, \sK_{2, \wt{X}})}/{p^{\infty}}$ follows from
the exact sequence
\begin{equation}\label{eqn:SK1-div-smooth-1}
0 \to {H^1(\wt{X}, \sK_{2, \wt{X}})}/{p^{\infty}} \to
{\underset{m}\varinjlim} \ H^1(\wt{X}, {\sK_{2, \wt{X}}}/{p^m}) 
\xrightarrow{\tau_{\wt{X}}}
{_{p^{\infty}}\CH^2(\wt{X})} \to 0
\end{equation}
and the Roitman torsion theorem for $\wt{X}$ from \cite{Milne}, whose proof 
involves showing that
$\tau_{\wt{X}}$ is an isomorphism. We refer to \cite[\S~6, p.~281]{Milne}
for this part of the argument in the proof of the Roitman torsion theorem.
\end{proof}

\begin{lem}\label{lem:SK1-rel-red}
${SK_1(\wt{X}, E)}/{p^{\infty}} = 0$.
\end{lem}
\begin{proof}
For $r \ge 1$, we have a commutative diagram
\[
\xymatrix@C.8pc{
& SK_1(\wt{X}, rE) \ar[r] \ar[d] & SK_1(\wt{X}) \ar[d] \ar[r] 
& SK_1(rE) \ar[d] \\
K_2(rE) \ar[r] & K_1(\wt{X}, rE) \ar[r] & K_1(\wt{X}) \ar[r] &
K_1(rE),}
\]
where the vertical arrows are injective.
Since $H^0(X, \sK_{1, (\wt{X}, rE)}) = 0$, 
we get an exact sequence
\begin{equation}\label{eqn:SK1-rel-red-0}
K_2(rE) \to SK_1(\wt{X}, rE) \to SK_1(\wt{X}) \to SK_1(rE).
\end{equation}
For $r = 1$, we break this exact sequence into the smaller exact sequences
\[
K_2(E) \to SK_1(\wt{X}, E) \to F_1 \to 0;
\]
\[
0 \to F_1 \to  SK_1(\wt{X}) \to F_2 \to 0;
\]
\[
0 \to F_2 \to SK_1(E) \to F_3 \to 0.
\]

If we let $\{E_1, \cdots , E_r\}$ denote the irreducible components
of $E$, we get maps $SK_1(E) \to \stackrel{r}{\underset{i = 1}\oplus} \
SK_1(E_i) \to (k^{\times})^{\oplus r}$, where the first map is the
sum of restrictions to components and the second map
is the sum of the push-forward maps $SK_1(E_i) \inj K_1(E_i) 
\to K_1(\Spec(k)) = k^{\times}$. It follows from 
\cite[Lemma~7.6]{KSri} that the composite map
$SK_1(E) \to (k^{\times})^{\oplus r}$ is isomorphism on the torsion subgroups.
On the other hand, it is shown in the proof of 
\cite[Lemma~7.7]{KSri} that the composite map
\[
\Pic(\wt{X}) \otimes_{\Z} k^{\times} \xrightarrow{\cup} SK_1(\wt{X}) \to
SK_1(E) \to (k^{\times})^{\oplus r}
\]
is surjective on the torsion subgroups.
It follows that restriction map $SK_1(\wt{X}) \to SK_1(E)$ induces a 
surjection ${_{p^{\infty}}SK_1(\wt{X})} \surj {_{p^{\infty}}SK_1(E)}$.

In particular, we get ${_{p^{\infty}}SK_1(\wt{X})} \surj {_{p^{\infty}}F_2}$,
which in turn shows that ${F_1}/{p^{\infty}} \inj
{SK_1(\wt{X})}/{p^{\infty}}$. \lemref{lem:SK1-div-smooth} implies 
that ${F_1}/{p^{\infty}} = 0$. To prove the lemma, we are now left with
showing that ${K_2(E)}/{p^{\infty}} = 0$.
Using the universal coefficient exact sequence
\begin{equation}\label{eqn:SK1-rel-red-1}
0 \to {K_n(E)}/{p^{\infty}} \to K_n(E; {\Z}/{p^{\infty}}) \to
{_{p^{\infty}}K_{n-1}(E)} \to 0,
\end{equation}
it suffices to show that $K_2(E; {\Z}/{p^{\infty}}) = 0$.

Let $S \subset E$ be a finite closed subset contained in the smooth locus
of $E$ such that $E \setminus S$ is affine. We then have 
the Thomason-Trobaugh localization exact sequence
\[
K^S_2(E; {\Z}/{p^m}) \to K_2(E; {\Z}/{p^m}) \to  
K_2(E \setminus S; {\Z}/{p^m}).
\]
Since $S$ lies in the smooth locus of $E$, we can replace 
$K^S_2(E; {\Z}/{p^m})$
by $K^S_2(E_{\rm reg} ; {\Z}/{p^m})$. On the other hand, as $E_{\rm reg}$
and $S$ are smooth, the maps $K_2(S; {\Z}/{p^m}) \to K'_2(S; {\Z}/{p^m})
\to K^S_2(E_{\rm reg} ; {\Z}/{p^m})$ are isomorphisms using the localization
sequence for the Quillen $K$-theory $K'(-)$. We thus have an exact sequence
\begin{equation}\label{eqn:extra*}
K_2(S; {\Z}/{p^m}) \to K_2(E; {\Z}/{p^m}) \to  K_2(E \setminus S; {\Z}/{p^m}).
\end{equation}

We now observe that ~\eqref{eqn:SK1-rel-red-1} holds even when we replace $E$ 
by $S$.  If we apply this for $S$ when $n = 2$, the
assumption that $k$ is algebraically closed of characteristic $p > 0$
implies that the first term of ~\eqref{eqn:extra*} is zero. 
We can therefore assume that $E$ is affine.

We let $E = \Spec(A), \ B = \Spec(A^N)$ and let $I \subset A$ denote the
ideal of $E_{\rm sing}$. Since $E$ is a strict normal crossing divisor
on a smooth surface, $I$ is in fact a conducting ideal for the normalization
map $A \to B$.
We now have a commutative diagram of exact sequences
\begin{equation}\label{eqn:SK1-rel-red-2}
\xymatrix@C.6pc{
K_4(B/I; {\Z}/{p^m}) \ar[r] & K_3(B, I; {\Z}/{p^m}) \ar[d] \ar[r] & 
K_3(B; {\Z}/{p^m}) & \\
& K_2(A,B,I; {\Z}/{p^m}) \ar[d] & & \\
K_3(A/I; {\Z}/{p^m}) \ar[r] \ar[d] & K_2(A, I; {\Z}/{p^m}) \ar[r] \ar[d] &
K_2(A; {\Z}/{p^m}) \ar[r] \ar[d] &  K_2(A/I; {\Z}/{p^m}) \ar[d] \\
K_3(B/I; {\Z}/{p^m}) \ar[r] & K_2(B, I; {\Z}/{p^m}) \ar[r] &
K_2(B; {\Z}/{p^m}) \ar[r] &  K_2(B/I; {\Z}/{p^m}).}
\end{equation}

Since $B, A/I, B/I$ are all smooth of dimension at most one, it follows
from \cite[Theorem~8.4]{GL} that $K_2(A,B,I; {\Z}/{p^m}) \xrightarrow{\simeq} 
K_2(A, I; {\Z}/{p^m}) 
\xrightarrow{\simeq} K_2(A; {\Z}/{p^m})$ for each $m \ge 1$.
It therefore suffices to show that $K_2(A,B,I; {\Z}/{p^{\infty}}) = 0$.

To prove this, we use the short exact sequence
\[
0 \to {K_2(A,B,I)}/{p^{\infty}} \to K_2(A,B,I; {\Z}/{p^{\infty}}) \to
{_{p^{\infty}}K_{1}(A,B,I)} \to 0.
\]
It follows from \cite[Theorem~C]{GH-1} (see \lemref{lem:p-fin-exp}) that 
$K_2(A,B,I)$ is a $p$-primary
torsion group of bounded exponent. Since ${\Q_p}/{\Z_p}$ is divisible,
we must have ${K_2(A,B,I)}/{p^{\infty}} = 0$.
On the other hand, 
$K_{1}(A,B,I) \simeq {I}/{I^2} \otimes_{B/I} \Omega^1_{{(B/I)}/{(A/I)}}$ by 
\cite[Theorem~0.2]{GW}. 
Since $B/I \simeq k^r$ for some $r \ge 1$, we have $\Omega^1_{{(B/I)}/{(A/I)}} =0$.
We conclude that $K_2(A,B,I; {\Z}/{p^{\infty}}) = 0$.
This finishes the proof.
\end{proof}

\begin{lem}\label{lem:SK1-rel-nonred}
For every $r \ge 1$, we have ${SK_1(\wt{X}, rE)}/{p^{\infty}} = 0$.
\end{lem}
\begin{proof}
The case $r = 1$ follows from \lemref{lem:SK1-rel-red}. So we assume
$r \ge 2$.
Using the long exact relative $K$-theory sequence, we get a commutative diagram 
\[
\xymatrix@C.8pc{
& SK_1(\wt{X}, rE) \ar[r] \ar[d] & SK_1(\wt{X}, E) \ar[r] \ar[d] &
SK_1(rE, E) \ar[d] \\
K_2(rE, E) \ar[r] & K_1(\wt{X}, rE) \ar[r] & K_1(\wt{X}, E) \ar[r] &
K_1(rE, E),}
\]
where the vertical arrows are all injective and the bottom row is exact.
It follows from ~\eqref{eqn:SK-1-K1} that the two vertical arrows from 
the left are isomorphisms. It follows that the top row is also exact.

Using \lemref{lem:p-fin-exp-nil}, we get an exact sequence
\[
0 \to A \to SK_1(\wt{X}, rE) \to SK_1(\wt{X}, E) \to B \to 0,
\]
where $A$ and $B$ are $p$-primary torsion groups of bounded exponents.
Setting $\ov{SK_1(\wt{X}, rE)}:= {SK_1(\wt{X}, rE)}/{A}$, the 
divisibility of ${\Q_p}/{\Z_p}$ implies that
\begin{equation}\label{eqn:SK1-rel-nonred-0}
{SK_1(\wt{X}, rE)}/{p^{\infty}} \xrightarrow{\simeq} 
{\ov{SK_1(\wt{X}, rE)}}/{p^{\infty}}.
\end{equation}

We now consider a commutative diagram of exact sequences
\[
\xymatrix@C.8pc{
0 \ar[r] & \ov{SK_1(\wt{X}, rE)} \ar[r] \ar[d] & 
SK_1(\wt{X}, E) \ar[r] \ar[d] & B \ar[d] \ar[r]  & 0 \\
0 \ar[r] & \ov{SK_1(\wt{X}, rE)}[p^{-1}] \ar[r] &
SK_1(\wt{X}, E)[p^{-1}] \ar[r] & B[p^{-1}] \ar[r] & 0,}
\]
where the vertical arrows are the localizations.

Since $B$ is $p$-primary torsion, the vertical arrow on the
right is zero. Combining this with \lemref{lem:SK1-rel-red} and the isomorphism 
${\Z[p^{-1}]}/{\Z} \simeq {\underset{m}\varinjlim}\ 
{\Z}/{p^m} \simeq {\Q_p}/{\Z_p}$, we get a surjection
$B \surj \ov{SK_1(\wt{X}, rE)} \otimes_{\Z} {\Z[p^{-1}]}/{\Z}
\simeq \ov{SK_1(\wt{X}, rE)} \otimes_{\Z} {\Q_p}/{\Z_p}$.
It follows that ${\ov{SK_1(\wt{X}, rE)}}/{p^{\infty}}$ is a divisible
group of bounded exponent, and hence it must be zero.
We conclude from ~\eqref{eqn:SK1-rel-nonred-0}
that ${SK_1(\wt{X}, rE)}/{p^{\infty}} = 0$.
\end{proof}

\begin{lem}\label{lem:div-H3K}
Let $X$ be a reduced quasi-projective surface over an algebraically closed
field $k$ of exponential characteristic $p \ge 1$. 
Assume that $X$ has only an isolated set of singular points.
Then $H^2(X, \sK_{3,X})$ is divisible by every integer $m$ prime to $p$.
In particular, ${H^2(X, \sK_{3,X})} \otimes_{\Z} {\Q_l}/{\Z_l} = 0$ for every
prime $l \neq p$.
\end{lem}
\begin{proof}
For any point $x \in X$ (not necessarily closed), 
let $X_x$ denote $\Spec(\sO_{X,x})$ and let
$X^o_x = X_x \setminus \{x\}$.
Let $X^{(i)}$ denote the set of codimension $i$ points on $X$.
For a Zariski sheaf $\sF$ on $X$ and a point $x \in X$, 
recall that $H^i_x(X, \sF)$
is defined as the colimit ${\underset{U}\varinjlim} \
H^i_{\ov{\{x\}} \cap U}(U, \sF|_U)$, where the limit is over all
open neighborhoods of $x$ in $X$.

The filtration by codimension of support gives rise to the coniveau spectral
sequence $E^{i,j}_1 = {\underset{x \in  X^{(i)}}\amalg} H^{i+j}_{x}(X, \sF)
\Rightarrow H^{i+j}(X, \sF)$. 
Since $H^{i}_{x}(X, \sF) \cong H^{i}_{x}(X_x, \sF)$
for every $x \in X$ and $i \ge 0$ by excision, and 
since $\dim(X^o_x) = i-1$ for every $x \in X^{(i)}$,
the above spectral sequence implies that the `forget support'
map ${\underset{x \in X^{(2)}}\oplus} H^2_x(X, \sF) \to H^2(X, \sF)$
is surjective. It suffices therefore to show that 
$H^2_x(X_x, \sK_{3,X})$ is $m$-divisible for every integer $m$ prime to $p$ and
every closed point $x \in X$. 

Suppose first that $x$ is a regular point. In this case,
the Gersten resolution for $\sK_{3,X}$ on $X_x$ implies
that $H^2_x(X_x, \sK_{3,X}) \simeq K_1(k(x)) \simeq k^{\times}$.
Our divisibility claim follows in this case.

Suppose next that $x$ is one of the isolated singular points of $X$.
In this case, we can use the localization exact sequence
$H^1(X^o_x,  \sK_{3,X}) \to H^2_x(X_x, \sK_{3,X}) \to
H^2(X_x, \sK_{3,X}) = 0$ to reduce the problem to showing that
$H^1(X^o_x,  \sK_{3,X})$ is $m$-divisibile.

Now, our assumption says that $X_x$ is a 2-dimensional local ring which 
is regular away from its closed point $x$. In particular, 
$X^o_x$ is regular and has dimension one.
We can therefore apply the Gersten sequence, which tells us that
there is a surjection ${\underset{x \in X^{(1)}_x}\oplus} K_2(k(x))
\surj H^1(X^o_x,  \sK_{3,X})$.
On the other hand, there is a Norm-residue isomorphism
${K_2(k(x))}/{m} \xrightarrow{\simeq} H^2_{\et}(k(x), \mu_m(2))$.
But the latter group vanishes since $k = \ov{k}$ and
${\rm tr. \ deg}_k(k(x)) = 1$.
We have thus shown that ${H^2_x(X_x, \sK_{3,X})}/m = 0$.
This finishes the proof.
\end{proof}

\begin{remk}\label{remk:higher-dim}
The reader can check from its proof that \lemref{lem:div-H3K} is valid
for $H^d(X, \sK_{d+1,X})$ for
any reduced quasi-projective scheme $X$ of dimension $d \ge 0$ over $k$ with
isolated singularities.
\end{remk}

\begin{thm}\label{thm:normal-vanishing}
Let $X$ be a normal projective surface over an algebraically closed
field $k$. Then
\[
{SK_1(X)} \otimes_{\Z} {\Q}/{\Z} = 0 = 
{H^1(X, \sK_{2,X})} \otimes_{\Z} {\Q}/{\Z}.
\]
\end{thm}
\begin{proof}
It suffices to show that ${SK_1(X)} \otimes_{\Z} {\Q}/{\Z} = 0$
as the remaining part follows from ~\eqref{eqn:PNormal-0-ex}.

Since ${\Q}/{\Z} \xrightarrow{\simeq} {\underset{l}\oplus} \ 
{\Q_l}/{\Z_l}$, where $l$ runs over the set of prime numbers,
it suffices to show that ${SK_1(X)}/{l^{\infty}} = 0$
for every prime $l$.
For $l \neq {\rm char}(k)$, this follows from 
Lemmas~\ref{lem:SK-1-exact}, ~\ref{lem:div-H3K} and \cite[Theorem~7.9]{BPW}.
So we assume ${\rm char}(k) = p > 0$ and show that
${SK_1(X)}/{p^{\infty}} = 0$.
In view of \lemref{lem:PNormal-SK1-A}, this is equivalent to
showing that ${SK_1(X, rY)}/{p^{\infty}} = 0$ for every 
$r \ge 1$.

To prove this, we consider a commutative diagram with exact bottom row
\[
\xymatrix@C.8pc{
& SK_1(X, rY) \ar[r] \ar@{^{(}->}[d] & SK_1(\wt{X}, rE) \ar@{^{(}->}[d] &
\\
K_1(X, \wt{X}, rY, rE) \ar[r] & K_1(X, rY) \ar[r] & K_1(\wt{X}, rE) \ar[r] &
K_0(X, \wt{X}, rY, rE).}
\]

Since $X$ is a normal surface, it is Cohen-Macaulay with isolated
singularities. It follows from ~\eqref{eqn:SK-1-K1} and
\thmref{thm:Desc-CM-Main-A} that there is an exact sequence
\[
0 \to A \to SK_1(X, rY) \to SK_1(\wt{X}, rE) \to B \to 0,
\]
where $A$ and $B$ are $p$-primary torsion groups of bounded exponents.
It follows from \lemref{lem:SK1-rel-nonred} that
${SK_1(\wt{X}, rE)}/{p^{\infty}} = 0$.
We now repeat the proof of \lemref{lem:SK1-rel-nonred} verbatim to 
conclude that ${SK_1(X, rY)}/{p^{\infty}} = 0$.
\end{proof}


\section{Torsion theorem for surfaces}
\label{sec:Non-affine}
We prove \thmref{thm:Intro-Res-1} in this section for surfaces.
In fact, we prove this torsion theorem for a certain class of non-affine
surfaces as well. This generalization will be needed in the proof of
the general case of \thmref{thm:Intro-Res-1} later in this paper.
We fix an algebraically closed field $k$ of any characteristic.
Recall from \S~\ref{sec:Prelim} that $\CH^q_F(X)$ denotes the
homological Chow group of codimension $q$ cycles on an equi-dimensional
scheme $X$. 
Recall from \cite{Levine-3} that every normal projective variety $X$ over
$k$ of dimension $d$ has an Albanese variety ${\rm Alb}(X)$ which is an abelian 
variety and there is an Albanese map $alb_X\colon \CH^d(X)_0 \to {\rm Alb}(X)$.

\subsection{Divisibility of $H^1(X, \sK_{2,X})$ for non-projective
normal surfaces}\label{sec:div-no-affine}
In order to prove our torsion theorem for surfaces, we need to 
extend \thmref{thm:normal-vanishing} to certain non-projective
normal surfaces. We do this below.

\begin{lem}\label{lem:surface-blow-up}
Let $X$ be a reduced affine surface over $k$ and let $\pi\colon Y \to X$
be the blow-up of $X$ along a closed subscheme $Z \subset X$ whose support
is finite. Assume that there is a closed subscheme $W \subset X$ such that 
$Z \cap W = \emptyset$ and $\pi^{-1}(X \setminus W)$ is smooth.
We can then find an open embedding $j\colon Y^N \inj \ov{Y}$, where
$\ov{Y}$ is a projective normal surface over $k$ such that the following hold.
\begin{enumerate}
\item 
$\ov{Y}_{\rm sing} = (Y^N)_{\rm sing}$.
\item
$\ov{Y} \setminus Y^N$ is a strict normal crossing divisor on $\ov{Y}$.
\item
The intersection of $T:= \ov{Y} \setminus Y^N$ with every connected component
of $\ov{Y}$ is connected.
\end{enumerate}
\end{lem}
\begin{proof}
We can assume $X$ to be connected.
Let $E \subset Y$ denote the reduced exceptional divisor of the
blow-up map $\pi$.
Let $f\colon X^N \to X$ be the normalization map and let $Z' = Z \times_X X^N$.
Consider the commutative square
\begin{equation}\label{eqn:surface-blow-up-0}
\xymatrix@C1pc{
{\rm Bl}_{Z'}(X^N) \ar[r]^/.8pc/{g} \ar[d]_{\pi^N} & Y \ar[d]^{\pi} \\
X^N \ar[r]_{f} & X.}
\end{equation}

It follows from \lemref{lem:Blow-up-fin} that $g$ is finite and birational.
The map $g^{-1}(Y \setminus E) \to Y \setminus E$ is clearly the 
normalization map.
On the other hand, it is given that $Y$ is smooth in an open neighborhood
of $E$. In particular, $g$ must be an isomorphism in that neighborhood.
It follows that ${\rm Bl}_{Z'}(X^N) = Y^N$.

Since $X^N$ is an affine normal surface, we can find an open embedding
$\ov{j}\colon X^N \inj \ov{X}$, where $\ov{X}$ is a normal projective surface over 
$k$ such that $\ov{X}_{\rm sing} = X^N_{\rm sing}$ and the complement
$T = \ov{X} \setminus X^N$ is a strict normal crossing divisor on $\ov{X}$. 
Moreover, it follows from \cite[Corollary~1]{Goodman} that the 
intersection of $T$ with every connected component
of $\ov{X}$ is connected.

Note that the existence of a normal projective surface $\ov{X}$ as
above is easily shown using the strong resolution of singularities 
techniques for surfaces in arbitrary characteristic. In brief, we first
choose a closed embedding $X^N \subset \A^n_k$ and take $Z$ to be the 
normalization of the closure of $X^N$ in $\P^n_k$. Using the strong
resolution of singularities techniques, one can now successively blow up 
0-dimensional closed subschemes lying over $Z \setminus X^N$
such that the composite map $p\colon\ov{X} \to Z$ has the property that
$\ov{X}$ is regular outside $X^N$ and $p^{-1}(Z \setminus X^N)$ is a strict 
normal crossing divisor.

Since $Z'$ is a closed subscheme of $X^N$ with finite support, it
is actually a closed subscheme of $\ov{X}$ and we get a Cartesian 
square
\begin{equation}\label{eqn:surface-blow-up-1}
\xymatrix@C1pc{
{\rm Bl}_{Z'}(X^N) \ar[r]^{j} \ar[d]_{\pi^N} & 
{\rm Bl}_{Z'}(\ov{X}) \ar[d]^{\ov{\pi}} \\
X^N \ar[r]_{\ov{j}} & \ov{X}.}
\end{equation}
As $\ov{\pi}$ is an isomorphism away from $Z'$, it follows that
${\rm Bl}_{Z'}(\ov{X})$ is normal outside $\ov{\pi}^{-1}(Z')$
and the singular locus of ${\rm Bl}_{Z'}(\ov{X})$ is same as
that of ${\rm Bl}_{Z'}(X^N)$.

On the other hand, the normality of $Y^N = {\rm Bl}_{Z'}(X^N)$ implies that
${\rm Bl}_{Z'}(\ov{X})$ is already normal in an open neighborhood of
$\ov{\pi}^{-1}(Z')$. It follows that ${\rm Bl}_{Z'}(\ov{X})$ is normal.
Furthermore, ${\rm Bl}_{Z'}(\ov{X}) \setminus {\rm Bl}_{Z'}(X^N)
= \ov{\pi}^{-1}(T)$ and the map $\ov{\pi}^{-1}(T) \to T$ is clearly an 
isomorphism. Since the connected components of ${\rm Bl}_{Z'}(\ov{X})$
are precisely the inverse images of the connected components of the
normal surface $\ov{X}$,
we see that the intersection of $\ov{\pi}^{-1}(T)$ with every connected 
component of ${\rm Bl}_{Z'}(\ov{X})$ is connected.
Setting $\ov{Y} = {\rm Bl}_{Z'}(\ov{X})$, we see that
$\ov{Y}$ is a normal projective surface with open embedding
$j \colon Y^N \inj \ov{Y}$ and $\ov{Y} \setminus Y^N = T$, satisfies the desired
properties.
\end{proof}

Let $X$ be a reduced affine surface over $k$ and let $\pi\colon Y \to X$
be as in \lemref{lem:surface-blow-up}. Let $j\colon Y^N \inj \ov{Y}$ be
the open embedding in a normal projective surface given by
\lemref{lem:surface-blow-up}. Let $i\colon T = \ov{Y} \setminus Y^N \inj \ov{Y}$ be
the inclusion of the strict normal crossing divisor on $\ov{Y}$.  
Using Lemma~\ref{lem:Ncd} and Bloch's formula for normal surfaces
(use either \cite[Theorems~2.2, 8.9]{PW2} or combine \cite{Levine-1} and 
\cite{BS-*})
$H^2(\ov{Y}, \sK_{2,\ov{Y}}) \cong \CH^2(\ov{Y})$, there is a push-forward map
$i_*\colon \CH^1_F(T) \to \CH^2(\ov{Y})$.

\begin{lem}\label{lem:Alb-blow-up}
The composite map
$\psi\colon \CH^1_F(T)_0 \xrightarrow{i_*} \CH^2(\ov{Y})_0 \to  \Alb(\ov{Y})$
is surjective. 
\end{lem}
\begin{proof}
We can assume $X$ to be connected.
Keeping the notations of \lemref{lem:surface-blow-up} and its proof, we let 
$\wt{\pi}\colon \wt{Y} \to \ov{Y}$ be a resolution
of singularities of $\ov{Y}$ such that the reduced exceptional divisor
$\wt{E}$ is strict normal crossing. One knows that the 
Albanese variety of a normal projective surface remains unchanged if one takes
its resolution of singularities (see \cite[\S~7]{KSri}), 
and thus one has an isomorphism
$\wt{\pi}_*\colon \Alb(\wt{Y}) \to \Alb(\ov{Y})$.
Let $\phi\colon \wt{Y} \to \Alb(\wt{Y})$ denote the universal morphism.
As $\ov{Y}_{\rm sing} \cap T = (Y^N)_{\rm sing} \cap T = \emptyset$, 
we can assume $T$ to be a closed subscheme of $\wt{Y}$.

Let $\ov{\pi} \colon \ov{Y} \to \ov{X}$ be as in ~\eqref{eqn:surface-blow-up-1}.
Since $X^N$ is smooth in an open neighborhood of $Z' = Z \times_X X^N$ by 
assumption, it follows that $F := \ov{\pi}^{-1}(Z')$  
is a closed subset of $\wt{Y}$ via $\wt{\pi}\colon \wt{Y} \to \ov{Y}$.
Furthermore, it is disjoint from $\wt{E}$ and $T$.
Set $D = \wt{E} + F + T$.

Since the intersection of $T$ with a connected component of $\wt{Y}$
is connected, it follows from \cite[Lemma~5.4]{MS} that
$\psi(\CH^1_F(T)_0)$ is an abelian subvariety of $\Alb(\ov{Y})$.
Let us denote this abelian subvariety by $B$. 

Following the proof of \cite[Theorem~5.3]{MS}, let $A = {\Alb(\ov{Y})}/B$
and let $\beta\colon \wt{Y} \to A$ denote the composite map
$\wt{Y} \to \Alb(\ov{Y}) \to A$. It is then clear that
$z := \beta(T)$ is a closed point of $A$. 
Since $A$ is divisible, we
need to show that the image of every irreducible component
of $\wt{Y}$ in $A$ is a point. We can thus assume that
$\wt{Y}$ is irreducible and, in particular, $T$ is connected.
We then need to show that $ \wt{Z}:= \beta(\wt{Y})$ is a point. 
Since $\beta$ is projective, $\beta(x) = z$ for every 
$x \in \wt{Y} \setminus D$ will imply  that 
$\dim(\wt{Z}) = 0$, and we are done. 
Otherwise, we can assume that there is a closed point 
$x \in \wt{Y} \setminus D$ such that $\beta(x) \neq z$.
We let $a = \beta(x)$ and $\wt{Y}_a = \beta^{-1}(a)$.

Suppose first that $\dim(\wt{Z}) = 1$. Since $\beta\colon \wt{Y} \to \wt{Z}$ 
is a 
projective and surjective morphism of relative dimension one,  each of its 
fibers is at least one-dimensional. 
In particular, $\dim(\wt{Y}_a) = 1$. Since $T \cap \wt{Y}_a = \emptyset$ and
$\wt{Y}_a \not\subset \wt{E}$, it follows that $\wt{\pi}(\wt{Y}_a)$ is a 
projective curve in $Y^N = \ov{Y} \setminus T$. 
Since $\wt{Y}_a \not\subset F$, it also follows that 
$\wt{\pi}(\wt{Y}_a)$ is a projective curve in $Y^N$ which is not contained
in $F$. Since $F$ is the exceptional divisor for the blow-up
map $\pi^N\colon Y^N \to X^N$, it follows that $\pi^N(\wt{\pi}(\wt{Y}_a))$
is a projective curve in $X^N$. But this is absurd as $X^N$ is affine.

Next assume that $\dim(\wt{Z}) = 2$. Then it follows from \cite[pg.~96]{Sef}
(or see pg.~71-72 of its online version)
that the intersection matrix of $\beta^{-1}(z)$ is negative definite. 
In particular, we get $T^2 \le 0$. On the other hand, we observe that
$T$ is a closed subscheme of $\ov{X}$ (in the notations of the proof of 
\lemref{lem:surface-blow-up}) with affine complement $X^N$. Hence, it 
follows from \cite[Theorem~1]{Goodman}
that there is a monoidal transformation $\eta\colon \ov{X}_1 \to \ov{X}$ with 
center in $T$ such that $\eta^{-1}(T)$ is the support of an ample divisor.
This implies in particular that we must have $T^2 > 0$, which is again
absurd. We have thus shown that $\dim(\wt{Z}) = 0$ and this finishes the
proof of the lemma.
\end{proof}

\begin{prop}\label{prop:surface-blow-up-div}
Let $X$ be a reduced affine surface over $k$ and let $\pi\colon Y \to X$
be as in \lemref{lem:surface-blow-up}. Then 
\[
H^1(Y^N, \sK_{2,Y^N}) \otimes_{\Z} {\Q}/{\Z} = 0 =
H^1(X^N, \sK_{2, X^N}) \otimes_{\Z} {\Q}/{\Z}.
\]
\end{prop}
\begin{proof}
We can assume $X$ to be connected.
Let $j\colon Y^N \inj \ov{Y}$ be the open embedding obtained in
\lemref{lem:surface-blow-up} with complement $\ov{Y} \setminus Y^N = T$.
We keep the notations of the proof of \lemref{lem:surface-blow-up}.

Using Bloch's formula for normal surfaces
(see the paragraph above \lemref{lem:Alb-blow-up}), 
there is a commutative diagram
\begin{equation}\label{eqn:affine-vanishing-0}
\xymatrix@C1pc{
\CH^2(\ov{Y}) \ar[r]^/-.1cm/{\simeq} \ar[d]_{j^*} &  H^2(\ov{Y}, \sK_{2,\ov{Y}}) 
\ar[d]^{j^*} \\
\CH^2(Y^N) \ar[r]_/-.3cm/{\simeq} & H^2(Y^N, \sK_{2, Y^N}),}
\end{equation}
where the two horizontal arrows are isomorphisms.
Since the left vertical arrow is surjective by the choice of $\ov{Y}$, it
follows that the right vertical arrow is also surjective.

We now consider the long exact sequence of Zariski cohomology groups
\[
H^1(\ov{Y}, \sK_{2,\ov{Y}}) \xrightarrow{j^*} H^1(Y^N, \sK_{2, Y^N})
\to H^2_T(\ov{Y}, \sK_{2,\ov{Y}}) \xrightarrow{i_*} 
H^2(\ov{Y}, \sK_{2,\ov{Y}})  \xrightarrow{j^*}
H^2(Y^N, \sK_{2,{Y^N}}).
\]

Lemma~\ref{lem:Ncd}, the isomorphisms of  
~\eqref{eqn:affine-vanishing-0} and the Gersten resolution for
$\sK_{2, \ov{Y}}$ (see \cite[Theorem~7]{Levine-1} for normal surfaces) 
together show that this exact sequence is
equivalent to an exact sequence
\[
H^1(\ov{Y}, \sK_{2,\ov{Y}}) \xrightarrow{j^*} H^1(Y^N, \sK_{2, Y^N}) \to
\CH^1_F(T) \xrightarrow{i_*} \CH^2(\ov{Y}) \xrightarrow{j^*} 
\CH^2(Y^N) \to 0.
\]

Since the normal surface 
$\ov{Y}$ is a disjoint union of its connected components 
$\{\ov{Y}_1,  \cdots , \ov{Y}_r\}$ and
the intersection of $T$ with each of these components is connected, 
we have a commutative diagram
\[
\xymatrix@C1pc{
H^1(Y^N, \sK_{2, Y^N}) \ar[r] & \CH^1_F(T) \ar[r]^{i_*} \ar@{->>}[d]_{\rm deg} &
\CH^2(\ov{Y}) \ar@{->>}[d]^{\rm deg} \\
& \Z^r \ar@{=}[r] & \Z^r.} 
\]

Recall that $\CH^1_F(T)_{0}$ and $\CH^2(\ov{Y})_{0}$ denote the
kernels of the degree maps. Since $\ov{Y}$ is smooth along $T$, it is easy to 
see that the map $\CH^2(\ov{Y})_{0} \to \CH^2(Y^N)$ is surjective.
We therefore get an exact sequence
\begin{equation}\label{eqn:affine-vanishing-1}
H^1(\ov{Y}, \sK_{2, \ov{Y}}) \xrightarrow{j^*} H^1(Y^N, \sK_{2, Y^N}) \to
\CH^1_F(T)_0 \xrightarrow{i_*} \CH^2(\ov{Y})_0 \xrightarrow{j^*} \CH^2(Y^N) 
\to 0.
\end{equation}

Let $\{T_1, \cdots , T_s\}$ denote the irreducible components of $T$.
Since each $T \cap \ov{Y}_i$ is connected and 
$T^N$ is smooth, the proper push-forward map 
$\Pic^0(T^N) \simeq \stackrel{s}{\underset{i =1}\oplus} \CH^1_F(T^N_i)_0 \to 
\CH^1_F(T)_0$ is surjective, by \cite[Lemma~5.1]{MS}.
Combining this with \lemref{lem:Alb-blow-up}, we see that the map
$T^N \to \ov{Y}$ induces a surjective morphism of abelian varieties
$\Pic^0(T^N) \surj \Alb(\ov{Y})$.  

We conclude from \lemref{lem:tor-groups} that the map
$\Pic^0(T^N)_{\rm tors} \to \Alb(\ov{Y})_{\rm tors}$ is surjective.
The factorization $\Pic^0(T^N)_{\rm tors} \to (\CH^1_F(T)_0)_{\rm tors} \to
\Alb(\ov{Y})_{\rm tors}$ implies that the composite map
$(\CH^1_F(T)_0)_{\rm tors} \to (\CH^2(\ov{Y})_0)_{\rm tors} \to 
\Alb(\ov{Y})_{\rm tors}$ is surjective.
On the other hand, the Roitman torsion theorem for projective normal surfaces
(see \cite[Theorem~1.6]{KSri}) says that the map
$\CH^2(\ov{Y})_0 \to \Alb(\ov{Y})$ is an isomorphism on the 
torsion subgroups. We conclude that the map
\begin{equation}\label{eqn:affine-vanishing-2}
(i_*)_{\rm tors}\colon (\CH^1_F(T)_0)_{\rm tors} \to (\CH^2(\ov{Y})_0)_{\rm tors}
\end{equation}
is surjective.

We set $H = {\rm Ker}(i_*)$ so that ~\eqref{eqn:affine-vanishing-1} yields 
a commutative diagram of short exact sequences
\begin{equation}\label{eqn:affine-vanishing-3}
\xymatrix@C1pc{
H^1(\ov{Y}, \sK_{2, \ov{Y}}) \ar[r]^{j^*} \ar[d] & 
H^1(Y^N, \sK_{2, Y^N}) \ar[r] \ar[d] & H \ar[r] \ar[d] & 0 \\
H^1(\ov{Y}, \sK_{2, \ov{Y}})_{\Q} \ar[r]_<<<{j^*} &
H^1(Y^N, \sK_{2, Y^N})_{\Q} \ar[r] &  H_{\Q} \ar[r] & 0,}
\end{equation}
where the vertical arrows are all localization maps.
Since $\Pic^0(T^N)$ is divisible and it surjects onto $\CH^1_F(T)_0$, it
follows that $\CH^1_F(T)_0$ is divisible. In particular, the map
$\CH^1_F(T)_0 \to (\CH^1_F(T)_0)_{\Q}$ is surjective.
Combining this with ~\eqref{eqn:affine-vanishing-1}  and
~\eqref{eqn:affine-vanishing-2}, we
conclude that the right vertical arrow in ~\eqref{eqn:affine-vanishing-3} is
surjective. The left vertical arrow in ~\eqref{eqn:affine-vanishing-3} is
surjective by \thmref{thm:normal-vanishing}. In particular, the
middle vertical arrow is also surjective. Equivalently,
$H^1(Y^N, \sK_{2, Y^N}) \otimes_{\Z} {\Q}/{\Z}  = 0$.
To prove that $H^1(X^N, \sK_{2, X^N}) \otimes_{\Z} {\Q}/{\Z}  = 0$, we
repeat the above proof verbatim, assuming that the map $\pi$ is identity.
This finishes the proof.
\end{proof}

\begin{lem}\label{lem:Tfree-normal}
Let $X$ be a reduced affine surface over $k$ and let 
$\pi\colon Y \to X$ be as in \lemref{lem:surface-blow-up}.
Then $\CH^2(X^N)$ and $\CH^2(Y^N)$ are uniquely divisible.
\end{lem}
\begin{proof}
We continue with the notations of the proof of 
\propref{prop:surface-blow-up-div}. We first prove the divisibility
assertion.
Since $\ov{Y}$ is a normal projective surface, Bertini theorem implies that
there is a smooth projective curve $C \subset \ov{Y}$ passing through
any given smooth point of $Y^N$ such that $C \cap (Y^N)_{\rm sing} = 
C \cap \ov{Y}_{\rm sing} = \emptyset$.
Any such curve defines a push-forward map $\Pic^0(C) \to {\CH^2(\ov{Y})}_0
\surj \CH^2(Y^N)$ (see, for instance, \cite[Lemma~1.8]{ESV}). 
It follows that $\CH^2(Y^N)$ is generated by the images of 
$\Pic^0(C) \to \CH^2(Y^N)$, where $C$ is a smooth projective curve. 
In particular, $\CH^2(Y^N)$ is divisible.
A similar argument works for $\CH^2(X^N)$ as well.

We now prove that $\CH^2(X^N)$ and $\CH^2(Y^N)$ are torsion-free.
We have shown in ~\eqref{eqn:affine-vanishing-1} that there is an exact sequence
\[
\CH^1_F(T)_0 \xrightarrow{i_*} \CH^2(\ov{Y})_0 \xrightarrow{j^*} 
\CH^2(Y^N) \to 0.
\]
We have also shown above that $\CH^1_F(T)_0$ is divisible. It follows that the
map ${(\CH^2(\ov{Y})_0)}_{\rm tors} \to$ \\
$\CH^2(Y^N)_{\rm tors}$ is surjective.
On the other hand, it is shown in ~\eqref{eqn:affine-vanishing-2} that
the map ${(\CH^1_F(T)_0)}_{\rm tors} \to {(\CH^2(\ov{Y})_0)}_{\rm tors}$ is 
surjective. We conclude that $\CH^2(Y^N)_{\rm tors} = 0$.

Finally, as $X^N$ is a normal affine surface over $k$, it follows from
\cite[Corollary~2.7]{Levine-2} and \cite[Corollary~1.7]{KSri}
that $\CH^2(X^N)$ is torsion-free.
\end{proof}

We now prove \thmref{thm:Intro-Res-1} for affine surfaces and their
blow-ups.

\begin{thm}\label{thm:Main-affine}
Let $X$ be a reduced affine surface over an algebraically closed field $k$.
Let $\pi\colon Y \to X$ be as in \lemref{lem:surface-blow-up}.
Then $\CH^2(X)$ and $\CH^2(Y)$ are uniquely divisible.
\end{thm}
\begin{proof}
We can assume that $X$ is connected.
Let $g\colon Y^N \to Y$ denote the normalization map.
By \cite[Theorem~3.3]{BPW} (see also \cite[Proposition~2.3]{Krishna-1} for
a refined version), there exists a conducting subscheme $W \subset Y$ with 
$W_{\rm red} = Y_{\rm sing}$ such that setting $\ov{W} = W \times_Y Y^N$, there
is a Mayer-Vietoris exact sequence
\begin{equation}\label{eqn:Main-affine-0}
SK_1(Y^N) \to \frac{SK_1(\ov{W})}{SK_1(W)} \to \CH^2(Y) \xrightarrow{g^*} 
\CH^2(Y^N) \to 0,
\end{equation}
where we have identified $SK_0(Y) = F^2K_0(Y)$ with $\CH^2(Y)$ (and
similarly for $Y^N$) by \cite[Theorem~7]{Levine-1}. 

It follows from our assumption that $W_{\rm red} = Y_{\rm sing}$ is affine.
We conclude that $W$ is affine (see \cite[Exercise~III.3.1]{Hartshorne-2}).
Since the map $\ov{W} \to W$ is finite, it follows that $\ov{W}$ is also affine.
It is classically known 
(see for instance, \cite[Chap.~IX, Propositions~1.3, 3.9]{Bass}) that 
$SK_1(R) = SK_1(R_{\rm red})$ for any commutative ring $R$. 
We can therefore assume that $W$ and $\ov{W}$ are reduced in 
~\eqref{eqn:Main-affine-0}.

It follows from the Thomason-Trobaugh spectral sequence 
$E^{p,q}_2 = H^p(X, \sK_{q,X}) \Rightarrow K_{q-p}(X)$ that there is a 
commutative diagram
\begin{equation}\label{eqn:Main-affine-1} 
\xymatrix@C1pc{
SK_1(Y^N) \ar[r] \ar[d] & H^1(Y^N, \sK_{2, Y^N}) \ar[d] \\
\frac{SK_1(\ov{W})}{SK_1(W)} \ar[r] & \frac{H^1(\ov{W}, \sK_{2, \ov{W}})}
{H^1(W, \sK_{2, W})},}
\end{equation}
where the top horizontal arrow is surjective (see \lemref{lem:SK-1-exact})
and the bottom horizontal arrow is an isomorphism.
We thus get an exact sequence
\begin{equation}\label{eqn:Main-affine-2} 
H^1(Y^N, \sK_{2, Y^N}) \to \frac{SK_1(\ov{W})}{SK_1(W)} \to \CH^2(Y) 
\xrightarrow{g^*} \CH^2(Y^N) \to 0.
\end{equation}

Since $W$ as well as $\ov{W}$ are reduced affine curves,
\cite[Theorem~5.3(i)]{BPW} tells us that there is an exact sequence
\[
SK_1(W) \to SK_1(\ov{W}) \to \frac{SK_1(\ov{W})}{SK_1(W)} \to 0,
\]
where the first two terms from the left are uniquely divisible.
It follows that the last term is also uniquely divisible.
We remark here that in the proof of \cite[Theorem~5.3(i)]{BPW}
(see Step~2 of the proof), the authors mistakenly write that 
if $Y$ is a reduced affine curve over $k$, then $SK_1(Y)$ is uniquely 
divisible (modulo a bounded $p$-torsion subgroup if $p \neq 0$). However,
their proof actually shows that $SK_1(Y)$ is always 
uniquely divisible, given the trivial observation that
every homomorphism from a torsion group to a torsion-free group is zero.

We now consider a commutative diagram with exact rows
\[
\xymatrix@C1pc{
H^1(Y^N, \sK_{2, Y^N}) \ar[r] \ar[d] & \frac{SK_1(\ov{W})}{SK_1(W)}
\ar[r] \ar[d] & {\rm Ker}(g^*) \ar[r] \ar[d] & 0 \\
H^1(Y^N, \sK_{2, Y^N})_{\Q} \ar[r] & {(\frac{SK_1(\ov{W})}{SK_1(W)})}_{\Q}
\ar[r] & {{\rm Ker}(g^*)}_{\Q} \ar[r] & 0.}
\]

We have just shown that the middle vertical arrow is an isomorphism.
It follows from \propref{prop:surface-blow-up-div} that the left vertical arrow
is surjective. We conclude that ${\rm Ker}(g^*)$ is uniquely divisible.

It follows from \lemref{lem:Tfree-normal} that $\CH^2(Y^N)$ is
uniquely divisible.
We conclude from ~\eqref{eqn:Main-affine-2} that $\CH^2(Y)$ is uniquely
divisible. 
Exactly the same argument, using \propref{prop:surface-blow-up-div} and
\cite[Theorem~1.6]{KSri}, also shows that $\CH^2(X)$ is uniquely
divisible.
\end{proof}


As an immediate consequence of \thmref{thm:Main-affine}, we obtain the
following extension of the affine Roitman torsion theorem to non-affine
surfaces. 

\begin{cor}\label{cor:KSri-extn}
Let $X$ be a reduced affine surface over an algebraically closed field $k$
and let $\pi\colon Y \to X$ be a resolution of singularities of $X$.
Then $\CH^2(Y)$ is uniquely divisible.
\end{cor}
\begin{proof}
Since $Y$ is also a resolution of singularities of the normalization $X^N$,
we can assume $X$ to be normal. The corollary now follows from
\thmref{thm:Main-affine} and
\cite[Remark~C, p.~155]{Lipman} (see \lemref{lem:Desingularization}).
\end{proof}

\section{Torsion theorem  in higher dimension}\label{sec:High-D}
In this section, we prove our torsion theorem for affine varieties
by reducing it to the case of surfaces using some blow-up techniques.
We shall use the following straight-forward commutative algebra result.

\begin{lem}\label{lem:red-finite} 
Let $R$ be commutative Noetherian ring such that $R_{\fp}$ is reduced 
for every associated prime ideal $\fp$ of $R$.
Then $R$ must be reduced.
\end{lem}
\begin{proof}
Let $a \neq 0$ be an element of $R$ such that $a^n = 0$ for some $n \ge 1$.
Let $\{\fp_1, \cdots , \fp_r\}$ be the set of associated primes of $R$
and set $S = R \setminus (\cup_i \fp_i)$. It follows from our 
assumption that $R_S$ is reduced. In particular,
$a$ dies in $R_S$. Equivalently, there exists $b \in S$ such that
$ab = 0$. But this means $a = 0$ because $b$ is not a zero-divisor in $R$. 
\end{proof}

We shall also use the following  algebro-geometric results
in the proof of \thmref{thm:Tor-High-D}. We leave the proof to the reader
as an exercise.

\begin{lem}\label{lem:Fulton*-0}
Let $X$ be a reduced quasi-projective scheme of dimension $d \ge 0$ over
a field $k$. For $r \ge 0$, 
let $\alpha, \beta \in \sZ^F_r(X)$ be such that their
supports are disjoint closed subsets of $X$.
Then $\alpha + \beta = 0$ in $\sZ^F_r(X)$ if and only if
$\alpha = 0 = \beta$.
\end{lem}

\begin{lem}\label{lem:rational}
Let $X$ be a smooth quasi-projective scheme of dimension $d \ge 2$ over
an algebraically closed field $k$. Let $\pi\colon X' \to X$ be a composition of
monoidal transformations with point centers and let $E \subset X'$ be
the exceptional divisor with $r$ irreducible components. 
Let $\beta \in \sZ^F_0(E)$ be a 0-cycle such that
${\rm deg}(\beta) = 0$ in $\Z^r$. 
Then we can find finitely many smooth projective
rational curves $L_j \subset E$ and rational functions $g_j \in k(L_j)$
such that $\beta = \sum_j (g_j)_{L_j}$.
\end{lem}
\begin{proof}
This is an easy exercise and we give only a sketch
(see the proofs of \cite[Lemma~5.2]{Bloch-1} and
\cite[Lemma~2.1]{Levine-3}).
We can write $E = \pi^{-1}(x)$, where $x \in X$ is a closed point.
Then $E$ is connected whose irreducible components $E_i$'s are 
successive point blow-ups of $\P^{d-1}_k$. In particular,
$\CH^F_0(E_i) \simeq \Z$. 
Let $\pi\colon E^N = \amalg_i E_i \to E$ denote the normalization map.

Recall now from \cite[\S~1.3]{Fulton} that $\sR^F_0(E)$ is generated by
the divisors of non-zero rational functions on integral curves on
$E$. Each of these curves must lie inside some irreducible
component $E_i$ of $E$. 
On the other hand, if $C \subset E_i$ is an irreducible curve and 
$f \in k(C)^{\times}$ is a rational function, then
$\pi(C) = C$ and $\pi_*(\divf(f)) = \divf(f)$. This shows that
the map $\pi_*\colon \sR^F_0(E^N) \to \sR^F_0(E)$ is surjective.

We also note that the map ${\rm deg}\colon \CH^F_0(E^N) \to  H^0(E^N, \Z)$
is an isomorphism and the map
$\pi_*\colon {\CH^F_0(E^N)}_0 \to {\CH^F_0(E)}_0$ is surjective
(see \cite[Lemma~5.1]{MS}). In particular, $\CH^F_0(E) 
\stackrel{\rm deg}{\to} \Z^r$ is an isomorphism so that 
$\sR^F_0(E) = {\rm Ker}(\sZ^F_0(E) \xrightarrow{\rm deg} \Z^r)$. 
We conclude that the map 
$\pi_*\colon \sR^F_0(E^N) \to {\rm Ker}(\sZ^F_0(E) \xrightarrow{\rm deg} \Z^r)$
is surjective.

Since $\pi$ takes a smooth connected rational curve isomorphically
onto its image, we have reduced the proof of the lemma to the case where 
$E \to \P^{d-1}_k$ 
is a composition of monoidal transformations with point centers.
In this case, the lemma is well known. 
\end{proof}

Let $X$ be a reduced quasi-projective scheme of dimension $d \ge 3$ over
an algebraically closed field $k$ and let $X \inj \P^n_k$ be a 
locally closed embedding. Let $Y \subset X$ be a nowhere dense closed
subset containing $X_{\rm sing}$ and let $W \subset X$ be a  
Cartier curve relative to $Y$. Let
$\ov{W}$ denote the scheme-theoretic closure of $W$ in $\P^n_k$ with the 
defining sheaf of ideals $\sI_{W}$.
Let $P$ be a property of hypersurface sections of $X$ in $\P^n_k$ which
contain $W$. Let $m, i \ge 1$ be two integers. 

\begin{defn}\label{defn:General} 
We shall say that a general choice of hypersurface $H_1$ in the linear system 
$|H^0(\P^n_k, \sI_W(m))|$ satisfies $P$ relative to $X$,
if there exists a dense open subset $\sU_1(W,m)
\subseteq |H^0(\P^n_k, \sI_W(m))|$ such that $X \not\subset H_1$ and
$X \cap H_1$ satisfies $P$, for every $H_1 \in \sU_1(W,m)$.

For $i \ge 2$, we shall say that a general choice of
hypersurfaces $\{H_1, \cdots , H_i\}$ in the linear system 
$|H^0(\P^n_k, \sI_W(m))|$ satisfies $P$ relative to $X$ if the
following hold.
\begin{enumerate}
\item
A general choice of hypersurface $H_1$ in the linear system 
$|H^0(\P^n_k, \sI_W(m))|$ satisfies $P$ relative to $X$ and $Y \cap H_1$
is a nowhere dense subset of $X \cap H_1$ containing $(X \cap H_1)_{\rm sing}$.

\item
For every $2 \le j \le i$ and every general choice of hypersurfaces
$\{H_1, \cdots , H_{j-1}\}$ in $|H^0(\P^n_k, \sI_W(m))|$,  (1)
holds for $X \cap H_1 \cap \cdots \cap H_{j-1}$.
\end{enumerate}
\end{defn}

Suppose that $P$ and $P'$ are two properties of the 
hypersurface sections of $X$ containing $W$.
Suppose that a general choice of hypersurfaces $\{H_1, \cdots , H_i\}$ 
in $|H^0(\P^n_k, \sI_W(m))|$ satisfies $P$ relative to $X$ and a 
general choice of hypersurfaces $\{H'_1, \cdots , H'_i\}$ in 
$|H^0(\P^n_k, \sI_W(m))|$ satisfies $P'$ relative to $X$.
It is then easy to see that a general choice of hypersurfaces 
$\{H_1, \cdots , H_i\}$ in $|H^0(\P^n_k, \sI_W(m))|$ will simultaneously
satisfy  $P$ and $P'$ relative to $X$.

\begin{lem}$(${\rm cf}. \cite[Lemma~1.3]{Levine-2}$)$\label{lem:Levine-Bertini}
For all $m$ sufficiently large, a general choice 
of hypersurfaces $\{H_1, \cdots , H_{d-1}\}$ in $|H^0(\P^n_k, \sI_W(m))|$ 
satisfies the following.
\begin{enumerate}
\item
If we let $W_j  =  X \cap H_1 \cap \cdots \cap H_j$, then
$W_j \subset X$ is a complete intersection along $W \cap Y$ for
$1 \le j \le d-1$.
\item
$W_j \setminus W \subset X \setminus W$ is a complete intersection 
closed subscheme which is reduced for $1 \le j \le d-2$.
\item
$\ov{(W_{d-1} \setminus W)} \cap W \cap Y = \emptyset$.
\end{enumerate}

In particular, $W \subset W_{d-2}$ is a complete intersection along 
$Y$.
\end{lem}
\begin{proof}
Let $\ov{X}$ be the closure of $X$ in $\P^n_k$ with the 
reduced structure and the defining sheaf of ideals $\sI_{\ov{X}}$.
Let $\sU_m(X)$ denote the complement of 
$|H^0(\P^n_k, \sI_{\ov{X}}(m))|$ in the linear system  
$|H^0(\P^n_k, \sI_W(m))|$.   
For $m \gg 0$, the restriction map
$\sU_m(X) \to |H^0(\ov{X}, {\sI_W}/{\sI_{\ov{X}}}(m))|$ is smooth and 
surjective with fibers $H^0(\P^n_k, \sI_{\ov{X}}(m))$. 
In particular, the inverse image of every open dense subset under this map
is open dense. We can therefore replace $\P^n_k$ by $\ov{X}$ in the 
proof of the lemma.

In this case, the assertions (1), (2) and (3) are 
straight-forward consequences of \cite[Lemma~1.3]{Levine-2} and its proof.
Levine actually proves a stronger version of (2).
He shows in a later part of his proof that $W_j$ is reduced for 
$1 \le j \le d-2$. In our application of this lemma in the proof of 
\thmref{thm:Tor-High-D} (see the choice of $Y'$ is Step ~4), 
we shall show that $W_{d-2}$ is reduced by a more 
refined choice of the hypersurfaces.

To prove the last assertion, note that the choice of $W_{d-1}$ inside $W_{d-2}$ 
is general, and hence the inclusion $W_{d-1} \subset W_{d-2}$ a complete
intersection along $W \cap Y$ by (1). On the other hand, it follows from (3) 
that $W \subset W_{d-1}$ is such that no irreducible component of $W_{d-1}$ 
which is different from a component of $W$, meets $W$ along $Y$.
In particular, at any point of $W \cap Y$, the ideal of
$W$ in $W_{d-2}$ coincides with that of $W_{d-1}$. 
Since $W \cap Y$ is finite, this ideal must be globally 
principal along $Y$. This finishes the proof.
\end{proof}

\begin{lem}\label{lem:Desingularization}
Let $X$ be a reduced quasi-projective surface over an algebraically closed
field $k$. Let $T \subset X$ be a finite set of closed points
such that the semi-local ring $\sO_{X,T}$ is normal.
Let $f\colon X' \to X$ be a proper map which is an isomorphism over 
$X \setminus T$.
Assume that $X'$ is regular over an open neighborhood of $T$ in $X$.
Then $f$ is either an isomorphism or, there exists a closed subscheme 
(possibly non-reduced) $Z \subset X$ supported on $T$ such that it is the 
blow-up of $X$ along $Z$.
\end{lem}
\begin{proof}
By throwing away points of $T$ where $f$ is an isomorphism,
we can reduce to the case where $f$ is not an isomorphism at any point of $T$.
Since the normality is an open condition on $X$, there is an affine
open neighborhood of $T$ in $X$ which is normal.
Let us write $T = T_1 \amalg T_2$ such that $\sO_{X,T_1}$ is regular and
$\sO_{X,T_2}$ is singular. It follows from \cite[Exercise~II.7.11]{Hartshorne-2}
that there exists a sheaf of ideals $\sI_1 \subset \sO_X$ supported on $T_1$
such that $f^{-1}(X \setminus T_2)$ is the blow-up of $X \setminus T_2$
along $\sI_1|_{(X\setminus T_2)}$.
Let $f_1\colon X_1 \to X$ be the blow-up of $X$ along $\sI_1$ and
let $T'_2 = f^{-1}_1(T_2)$. We then have a factorization
$X' \xrightarrow{f_2} X_1 \xrightarrow{f_1} X$ of $f$.

Since $X_1$ is normal in a neighborhood of the finite set of closed
points $T'_2$ and, since $X'$ is regular over $T'_2$, it follows from 
\cite[Remark~C, p.~155]{Lipman} that there exists a sheaf of ideals
$\sI'_2 \subset \sO_{X_1}$ supported on $T'_2$ such that $f_2$ is the blow-up
of $X_1$ along $\sI'_2$. Since $f_1$ is an isomorphism over $T_2$,
one checks that $(f_1)_*(\sI_2')$ is a sheaf of ideals on $X$.
We let $\sI_2 = (f_1)_*(\sI_2')$ and $\sI = \sI_1 \cdot \sI_2$. 
If $Z$ is the closed subscheme of $X$ defined by $\sI \subset \sO_X$,
it becomes clear that it is supported on $T$ and,
$f$ is the blow-up of $X$ along $Z$.
\end{proof}

\begin{thm}\label{thm:Tor-High-D}
Let $X$ be a reduced affine scheme of dimension $d \ge 2$ over
an algebraically closed field $k$. 
Then $\CH^d(X)$ is uniquely divisible.
\end{thm}
\begin{proof}
In view of \thmref{thm:Main-affine}, we can assume $d \ge 3$.
To show that $\CH^d(X)$ is divisible, we let $\alpha \in \sZ^d(X, X_{\rm sing})$
be a 0-cycle. We fix a locally closed embedding $X \inj \P^N_k$.
It follows from \lemref{lem:Levine-Bertini} and \cite[Theorem~7]{KL} that
there exists a complete intersection reduced affine surface
$Y \subset X$ containing the support of $\alpha$ such that
$Y$ is smooth away from $X_{\rm sing}$.  


Since the inclusion
$i\colon Y \inj X$ is a complete intersection
and $Y$ is smooth away from $X_{\rm sing}$, it follows that $X_{\rm sing}
\cap Y = Y_{\rm sing}$. We conclude from ~\lemref{lem:PF} that there exists a 
push-forward map $i_*\colon \CH^2(Y) \to \CH^d(X)$ which contains $\alpha$
in its image. Since $\CH^2(Y)$ is divisible by
\thmref{thm:Main-affine}, we conclude that $\CH^d(X)$ is divisible.
The rest of the proof will be devoted to showing that $\CH^d(X)$
is torsion-free.

Let $\alpha \in \sZ^d(X, X_{\rm sing})$ be such that $n\alpha = 0$ in
$\CH^d(X)$ for some integer $n \ge 1$. Equivalently,
$n\alpha \in \sR^d(X, X_{\rm sing})$.
We can thus use \lemref{lem:MLemma} to find a (reduced) Cartier curve
$C$ on $X$ and a function $f \in \sO^{\times}_{C, S}$ such that $n\alpha = 
(f)_{C}$, where $S = C \cap X_{\rm sing}$.
Since that part of $n\alpha$ supported on any connected component of
$C$ is also of the form
$n\alpha' = (f')_{C'}$ for some Cartier curve $C'$ and some $f'$, 
we can assume that $C$ is connected. Let $\{C_1, \cdots , C_r\}$ denote
the set of irreducible components of $C$. 
We set $U = X \setminus X_{\rm sing}$ so that $U$ is smooth and a dense open 
subscheme of $X$.

Our aim is to show that $\alpha \in \sR^d(X, X_{\rm sing})$, by
reducing to the case where $X$ is replaced by a surface.
Ideally, one would like to find a complete intersection surface
in $X$ which supports $\alpha$ and on which, $n\alpha$ is rationally equivalent 
to zero. Since this can not be achieved in general, the strategy is to use 
Bloch's trick \cite{Bloch-1} 
of finding such a surface after a series of monoidal
transformations of $X$. What makes this trick work in our case is 
a combination of Lemma~\ref{lem:PF} and \thmref{thm:Main-affine}.
This strategy is executed as follows. 

\vskip .3cm

{\bf STEP~1}: Following an argument of Bloch \cite[Chapter~5]{Bloch-1}, we
let $\pi \colon X' \to X$ be a successive blow-up at smooth points such that the
following hold.
\begin{enumerate}
\item
The strict transform $D_i$ of each $C_i$ is smooth at every point
of $D_i \cap \pi^{-1}(U)$.
\item
$D_i \cap D_j \cap \pi^{-1}(U) = \emptyset$ for $i \neq j$.
\item
Each $D_i$ intersects the exceptional divisor $E$ (which is reduced) 
transversely at smooth points of $X'$.
\end{enumerate}

It is clear that there exists a finite set of distinct
blown-up closed points
$T = \{x_1, \cdots , x_n\} \subset U$ such that 
$\pi\colon \pi^{-1}(X\setminus T) \to X \setminus T$
is an isomorphism. In particular, $\pi\colon X'_{\rm sing} \to X_{\rm sing}$
is an isomorphism. This in turn implies that $X'_{\rm sing}$ is an affine
scheme of dimension at most $d-1$. Set $U' = X' \setminus X'_{\rm sing}
= \pi^{-1}(U)$. Let $D$ denote the strict transform of $C$ with
components $\{D_1, \cdots, D_r\}$.

Since $\pi$ is an isomorphism over an open neighborhood of $X_{\rm sing}$,
it follows that $S' := \pi^{-1}(S) \simeq S$ and the map $D \to C$ is an
isomorphism along $X_{\rm sing}$. In particular, 
$\pi^*\colon \sO_{D, S} \xrightarrow{\simeq} \sO_{D', S'}$
and $\pi_*((\pi^*(f))_{D}) = (f)_{C} = n\alpha$. 
In the rest of the proof, we shall often abuse the
notations by identifying 
$S, \ \sO_{D, S}$ and $f$ with $S', \ \sO_{D', S'}$ and $\pi^*(f)$,
respectively. 
Since ${\rm Supp}(\alpha) \subset C$, we can find
$\alpha' \in \sZ^d(X', X'_{\rm sing})$ supported on $D$ 
such that $\pi_*(\alpha') = \alpha$.
This implies that 
$\pi_*(n\alpha' - (f)_{D}) = 0$.
Setting $\beta = n\alpha' - (f)_{D}$, we get $\pi_*(\beta) = 0$ in the
cycle group $\sZ^d(X, X_{\rm sing})$.

\vskip .2cm

{\bf STEP~2}:
We can now write $\beta = \stackrel{n}{\underset{i =0}\sum} \beta_i$,
where $\beta_i$ is a 0-cycle on $X'$ supported on $\pi^{-1}(x_i)$
for $1 \le i \le n$ and $\beta_0$ is supported on the complement of the
exceptional divisor $E$. We then get 
$\stackrel{n}{\underset{i =0}\sum} \pi_*(\beta_i) = 0$ in 
$\sZ^d(X, X_{\rm sing}) \subseteq \sZ^d_F(X)$.  
Since all closed points of $T$ are distinct and the support of $\pi_*(\beta_0)$
is disjoint from $T$, a repeated application of \lemref{lem:Fulton*-0} 
shows that $\pi_*(\beta_i) = 0$
for all $0 \le i \le n$. Since $\pi$ is an isomorphism away from 
$T = \{x_1, \cdots , x_n\}$, we must have $\beta_0 = 0$.
We can therefore assume that $\beta$ is a 0-cycle on $E$.

We now note that each $\pi^{-1}(\{x_i\})$ is a $(d-1)$-dimensional projective 
scheme whose irreducible components are successive point blow-ups of 
$\P^{d-1}_k$. Moreover, we have 
$\pi_*(\beta_i) = 0$ for the push-forward map 
$\pi_*\colon \sZ^d_F(\pi^{-1}(\{x_i\})) \to \Z$, induced by 
$\pi\colon \pi^{-1}(\{x_i\})
\to \Spec(k(x_i)) \xrightarrow{\simeq} \Spec(k)$. 
But this means that ${\rm deg}(\beta_i) = 0$.
Taking the sum, we get
${\rm deg}(\beta) = \stackrel{n}{\underset{i =1}\sum} {\rm deg}(\beta_i) = 0$.
An application of \lemref{lem:rational} now shows that 
there are finitely many smooth projective rational curves 
$L_j \subset E$ and rational functions $g_j \in k(L_j)$ such that
$\beta = \stackrel{s}{\underset{j = 1}\sum}  (g_j)_{L_j}$.

\vskip .2cm

{\bf STEP~3}: Using the next step in the 
argument of Bloch (see \cite[Lemma~5.2]{Bloch-1}), after possibly further
blow-up of $X'$ along the closed points supported on $E$, we can assume
that no more than two $L_j$'s meet at a point and they intersect
$D$ transversely (note that $D$ is smooth along $E$).
In particular, in combination with (1) - (3) above, this implies that
$D':= D \cup (\cup_j L_j)$ is a 
reduced curve with
following properties (see line 4 from the bottom of \cite[p.~5.2]{Bloch-1}).

\begin{enumerate}
\item
Each component of $D' \cap U'$ is smooth.
\item
$D' \cap U'$ has only ordinary double point singularities, i.e., 
exactly two components of $D' \cap U'$ meet at any of its 
singular points with distinct tangent directions. 
\end{enumerate}

In particular, the embedding 
dimension of $D'$ at each of its singular points lying over $U$ is two.
Furthermore, $D' \cap X'_{\rm sing} = 
(D' \setminus (\cup_j L_j)) \cap X'_{\rm sing} = D \cap X'_{\rm sing}$.
This implies that $D'$ is a Cartier curve on $X'$.

\vskip .2cm

{\bf STEP~4}: We now fix a locally closed embedding $X' \inj \P^n_k$.
Since $X'$ is reduced of dimension $d \ge 3$, 
and since $D' \subset X'$ is a local complete intersection along 
$X'_{\rm sing}$, \lemref{lem:Levine-Bertini} applies. 
If for $m \gg 0$, we let $\{H_1, \cdots , H_{d-1}\}$ be a 
general choice of hypersurfaces in $|H^0(\P^n_k, \sI_{D'}(m))|$ 
satisfying \lemref{lem:Levine-Bertini}, then the scheme-theoretic
intersection $Y' = X' \cap H_1 \cap \cdots \cap H_{d-2}$ has the
property that $Y' \subset X'$ is a
complete intersection away from $D'$ and along $D' \cap X'_{\rm sing}$.
Furthermore, $Y' \setminus D'$ is reduced and
$D' \subset Y'$ is a complete intersection along $X'_{\rm sing}$.

We shall now refine the choice of $Y'$ using the Bertini theorems
of Altman and Kleiman \cite{KL}. 
We write $W':= D'\cap U'$ (recall that $U' = X' \setminus X'_{\rm sing}$)
and let ${W'}(\Omega^1_{D'}, e)$ denote the 
locally closed subset of points in $W'$ where the embedding dimension
of $W'$ is $e$. It follows from the above description of $D'$
that ${\underset{e}{\rm max}} \{\dim({W'}(\Omega^1_{D'}, e)) + e\} \le 2$.
In this case, \cite[Theorem~7]{KL} says that for all $m \gg 0$, 
a general choice of hypersurfaces 
$\{H_1, \cdots , H_{d-2}\}$ in $|H^0(\P^n_k, \sI_{D'}(m))|$
has the property that $U' \cap H_1 \cap \cdots \cap H_j$ is smooth
for $1 \le j \le d-2$.   
We conclude that for all $m \gg 0$ and a general choice of 
hypersurfaces $\{H_1, \cdots , H_{d-2}\}$ in $|H^0(\P^n_k, \sI_{D'}(m))|$,
the scheme-theoretic intersection $Y' = X' \cap H_1 \cap \cdots \cap H_{d-2}$
has the following properties.
\begin{enumerate}
\item
$\dim(Y') = 2$.
\item
$Y' \subset X'$ is a complete intersection away from $D'$ and along
$S'$.
\item
$Y'$ is reduced away from $S'$.
\item
$D' \subset Y'$ is a local complete intersection along $X'_{\rm sing}$.
\item
$Y' \cap U'$ is smooth.
\end{enumerate}

For such a surface $Y'$, (4) says that its local rings
at $D' \cap X'_{\rm sing} \ (= S')$ contain regular elements. 
In particular, $S'$ contains no embedded prime of $Y'$. We
conclude from (3) and \lemref{lem:red-finite} that this surface $Y'$
must be reduced. Furthermore, a combination of (2)  and (5) shows that
$Y' \subset X'$ must be a complete intersection.
Let $\iota\colon Y' \inj X'$ denote this inclusion.

\vskip .2cm

{\bf STEP~5}:
If we let $Z' = Y' \cap X'_{\rm sing} = (\pi \circ \iota)^{-1}(X_{\rm sing})$, 
it follows from the above constructions that $Y'_{\rm sing} \subseteq Z'$ and
$\dim(Z') \le 1$. The cycle $\alpha'$ lies in $\sZ^2(Y', Z')$
under the inclusion $\iota_*\colon \sZ^2(Y', Z') \inj \sZ^d(X', X'_{\rm sing})$
so that we can write $\alpha' = \iota_*(\alpha')$. Moreover, 
we have $n\alpha' = (f)_D + \sum_j (g_j)_{L_j}$ in $\sZ^2(Y', Z')$,
where $D$ and $L_j$'s are Cartier curves in $Y'$.
In particular, $n\alpha' = 0$ in $\CH^2(Y')$.

Our final goal is to apply \thmref{thm:Main-affine} to
show that $\alpha'$ dies in $\CH^2(Y')$. Before we do so in the next step, we
observe here that no component of $Y'$ can be contained in $E$. 
In particular, the map $Y' \to \pi(Y')$ is birational.
Indeed, as $T \subset C = \pi(D') = \pi(D)$, the exceptional divisor $E$ can
intersect only those components of $U' = X' \setminus X'_{\rm sing}$ 
which meet $D$.
On the other hand, if $U'_i$ is a connected component of $U'$ which meets $D$,
then we must have $\dim(D \cap U'_i) = 1$ (since $D \cap X'_{\rm sing}$ is 
finite) and $\dim(D \cap U'_i\cap E) = 0$.
Since $Y' \cap U'_i$ is smooth connected and contains $D \cap U'_i$,
the claim follows. 

\enlargethispage{20pt}
\vskip .2cm

{\bf STEP~6}:
We let $Y$ denote the scheme-theoretic image of the map $Y' \to X$.
By definition, $Y$ is the closed subscheme of $X$ determined by the sheaf of
ideals $\sI_Y = {\rm Ker}(\sO_X \to \pi_*(\sO_{Y'}))$.
Since $\pi$ is proper and $Y'$ is reduced, it follows from \cite[Tag~01R5]{SP}
that $Y$ coincides with $\pi(Y')$ with the reduced induced scheme structure.
One can also see directly that $Y$ is reduced because there is an
injection $\sO_Y \simeq {\sO_X}/{\sI_Y} \inj \pi_*(\sO_{Y'})$ and
the latter is clearly a sheaf of reduced rings on $X$ since $\sO_{Y'}$ is
a sheaf of reduced rings on $X'$.
Let $\pi'\colon Y' \to Y$ denote the restriction of $\pi$ to $Y'$
and let $W = \pi'(Y'_{\rm sing})$.

Recall from STEP~1 that $T$ is the finite set of closed points
in $C \subset X$ such that $\pi$ is an isomorphism away from $T$.
Since $T \subset \pi(D)$ and $D \subset Y'$, it follows that $T \subset Y$.
Since $X'_{\rm sing} \xrightarrow{\simeq} X_{\rm sing}$ and $Y' \setminus 
X'_{\rm sing}$ is smooth by STEP~4, it follows that
$W \subset Y$ is a closed subset such that
$W \cap T = \emptyset$ and $\pi'^{-1}(Y \setminus W)$ is smooth. 
Since $Y$ is affine, we conclude that $\pi'\colon Y' \to Y$ is a projective 
birational morphism of reduced quasi-projective surfaces 
which satisfies all conditions of \lemref{lem:surface-blow-up},
except that we do not know if $\pi'$ is a blow up with center supported on $T$.
One can show that this is indeed the case, but we shall not use this here. 
We use the following trick instead.

We let $Y_1 = \Spec(\pi'_*(\sO_{Y'}))$ and let $Y' \xrightarrow{p_1} 
Y_1 \xrightarrow{p_2} Y$ be the Stein factorization of $\pi'$
(see \cite[Corollary~III.11.5]{Hartshorne-2}).
Since $Y'$ is regular over a neighborhood of $T$, it is easy to check that 
$Y_1 \xrightarrow{p_2} Y$ coincides with the 
normalization of $Y$ at all points of $T$ 
(see \cite[Remark~D, p.~155]{Lipman}). 
Since $p_2$ is finite, $Y_1$ must be affine. We let $W_1 = p^{-1}_2(W)$
and $T_1 = p^{-1}_2(T)$. 


We now have a projective birational morphism $p_1\colon Y' \to Y_1$,
where $Y_1$ is a reduced affine surface
with disjoint closed subsets $W_1$ and $T_1$ such that $T_1$ is finite
and $\sO_{Y_1, T_1}$ is normal. Moreover, $p_1$ is isomorphism over 
$Y_1 \setminus T_1$ (because $\pi'$ is) and $p^{-1}_1(Y_1 \setminus W_1)$
is smooth. We conclude from \lemref{lem:Desingularization} that
$p_1$ is either an isomorphism or satisfies all conditions of 
\lemref{lem:surface-blow-up}. 

In both cases, it follows from \thmref{thm:Main-affine} that 
$\alpha' = 0$ in $\CH^2(Y')$.
On the other hand, as $Y' \subset X'$ is a complete intersection,
every singular point of $X'$ along $Y'$ must be a singular point
of $Y'$ as well. It follows that $Z'$ coincides with $Y'_{\sing}$ as
a closed subset of $Y'$.
We conclude that $\alpha' = 0$ in $\CH^2(Y', Z')$.
Finally, it follows  from \lemref{lem:PF} that 
$\alpha = \pi_* \circ \iota_*(\alpha')= 0$. 
This finishes the proof of the theorem.
\end{proof}

\section{Proof of Murthy's conjecture}\label{sec:ECG}
Let $k$ be an algebraically closed field of any characteristic.
In this section, we shall apply \thmref{thm:Tor-High-D} to 
prove Murthy's conjecture, and to 
show that the Euler class groups of affine $k$-algebras 
coincide with their Levine-Weibel Chow groups.
We shall also give the proof of \corref{cor:Intro-Res-5}.

The Euler class groups of affine algebras over a field was invented by Nori and 
extensively studied by Bhatwadekar-R. Sridharan.
The most notable feature of the Euler class group $E(A)$ is that 
every projective $A$-module of top rank has an Euler class in $E(A)$,
and the vanishing of this class characterizes the existence of a rank one
free summand of the projective module. 

An open question of Bhatwadekar-R. Sridharan \cite[Remark~3.3]{BS-4} asks
if the weak Euler class groups of smooth algebras
coincide with the Chow groups of
0-cycles. Over an algebraically closed field, this was
shown by Bhatwadekar-R. Sridharan \cite[Corollary~4.15]{BS-1}.
We shall use \thmref{thm:Tor-High-D} to give a direct extension of this
result to arbitrary affine algebras.

\subsection{Euler and weak Euler class groups}
Let $A$ be a reduced affine algebra over $k$ of dimension $d \ge 2$.
We recall the following definition from \cite[\S~2]{BS-2}.

\enlargethispage{20pt}

\begin{defn}\label{defn:EGC-0}
Let $F(A)$ be the free abelian group on the set of pairs $(J, \omega_J)$, where:
\begin{enumerate} 
\item
$J$ is an ideal of $A$ of height $d$.
\item
$J$ is $\fm$-primary for some maximal ideal $\fm$ of $A$.
\item
$\omega_J\colon (A/J)^d \surj J/J^2$ is a surjective map of $A/J$-modules.
\end{enumerate}

Let $[(J, \omega_J)]$ denote the class of $(J, \omega_J)$ in $F(A)$.
Given an ideal $I$ of height $d$, let $I = \cap_i J_i$ be an
irredundant primary decomposition of $I$, where each $J_i$ is $\fm_i$-primary 
for some maximal ideal $\fm_i$ of $A$. 
If $\omega_I\colon  (A/I)^d \surj I/{I^2}$ is a surjection, then
the Chinese remainder theorem shows that $\omega_I$ uniquely
defines a surjection $\omega_{J_i}\colon (A/{J_i})^d \surj {J_i}/{J^2_i}$.
In particular, each $[(J_i, \omega_{J_i})] \in F(A)$.
We write $[(I, \omega_I)] = \sum_i[(J_i, \omega_{J_i})] \in F(A)$.
Note that this is well defined because the irredundant primary decomposition
$I = \cap_i J_i$ is unique in our case by \cite[Theorem~4.10]{AMac}.

Let $K(A)$ be the subgroup of $F(A)$ generated by the elements of the type
$[(I, \omega_I)]$, where $I \subset A$ is an ideal of height $d$ with an 
$A$-linear surjection $\wt{\omega}_I\colon A^d \surj I$
such that the diagram 
\[
\xymatrix@C1pc{
A^d \ar@{->>}[r]^{\wt{\omega}_I} \ar[d] & I \ar[d] \\
(A/I)^d \ar@{->>}[r]_{\omega_I} & I/{I^2}}
\]
commutes.
The quotient $E(A):= {F(A)}/{K(A)}$ is called the {\sl Euler class group} of 
$A$.
\end{defn}

\begin{defn}\label{defn:WEGC-0}
Let $G(A)$ be the free abelian group on the set of ideals $J \subset A$
such that:
\begin{enumerate} 
\item
$J$ is of height $d$.
\item
$J$ is $\fm$-primary for some maximal ideal $\fm$ of $A$.
\item
There is a surjection of $A/J$-modules $\omega_J\colon (A/J)^d \surj J/J^2$.
\end{enumerate}

Let $[J]$ denote the class of $J$ in $G(A)$.
Given an ideal $I$ of height $d$, let $I = \cap_i J_i$ be an
irredundant primary decomposition of $I$, where each $J_i$ is $\fm_i$-primary 
for some maximal ideal $\fm_i$ of $A$. 
If $I/{I^2}$ is generated by $d$ elements,
the Chinese remainder theorem shows that each $[J_i] \in G(A)$.
We write $[I] = \sum_i[J_i] \in G(A)$.

Let $H(A)$ be the subgroup of $G(A)$ generated by the elements of the type
$[I]$, where $I \subset A$ is an ideal of height $d$ with an $A$-linear
surjection $\omega_I\colon A^d \surj I$. The quotient ${G(A)}/{H(A)}$ is called the
{\sl weak Euler class group} of $A$ and is denoted by $E_0(A)$.
\end{defn}

There is an evident surjection $E(A) \surj E_0(A)$. But the following
stronger result holds. 

\begin{lem}\label{lem:ECG-SW}
Let $A$ be a reduced affine algebra of dimension $d \ge 2$
over an algebraically closed field $k$.
Then the map $E(A) \to E_0(A)$ is an isomorphism.
\end{lem}
\begin{proof}
This is a well known consequence of various results of
\cite{BS-1}, \cite{BS-2} and \cite{BS-3}. We only give a sketch.
It follows from \cite[Lemma~3.3]{BS-2} (see also the proof of
\cite[Corollary~7.9]{BS-3}) that the kernel of the 
map $E(A) \surj E_0(A)$ is generated by the elements of the type
$(I, \omega_I)$, where $I = (a_1, \cdots , a_d)$.

Hence, we need to show that if $I$ is an ideal of height $d$
generated by $d$ elements and if $\omega_I\colon  (A/I)^d \surj I/I^2$ is a given
surjection, then $[(I,\omega_I)] = 0$ in $E(A)$.
But this is proven in \cite[Proposition~5.1]{Das}.
This result is stated in {\sl loc. cit.} for fields of characteristic
zero, but the proof works without any change for any algebraically closed
field. This is because the only two results needed in the proof are
Suslin's cancellation theorem \cite{Suslin} and \cite[Lemma~2.4]{Srid}, 
both of which hold for reduced affine algebras 
when the base field is algebraically closed of any characteristic.
\end{proof}

\subsection{Euler class group and Chow group}
To connect the Euler class group with the Chow group of 0-cycles, we
define a variant of the weak Euler class group as follows. 
When $A$ is a smooth, this variant was considered  
by Bhatwadekar-R. Sridharan \cite[Definition~2.2]{BS-4}.

Let $A$ be a reduced affine algebra of dimension $d \ge 2$
over an algebraically closed field $k$. Let $G_s(A)$ denote the
free abelian group on the set of smooth maximal ideals of $A$.
We shall say that an ideal $I \subset A$ of height $d$ is {\sl regular}
if it is of the form
$I = \cap_i \fm_i$, where each $\fm_i$ is a smooth maximal ideal with
$\fm_i \neq \fm_j$ for $i \neq j$.  
Given a regular ideal $I \subset A$, we set $[I] = \sum_i [\fm_i] \in
G_s(A)$. Let $H_s(A)$ denote the subgroup of $G_s(A)$ generated by
elements $[I]$ as above such that $I$ is a complete intersection in $A$.
We set $E_s(A) = {G_s(A)}/{H_s(A)}$.

There is an evident group homomorphism $\phi_A\colon E_s(A) \to E_0(A)$.

\begin{lem}\label{lem:MS-Bertini}
The map $\phi_A\colon  E_s(A) \to E_0(A)$ is an isomorphism.
\end{lem}
\begin{proof}
We prove the surjectivity of $\phi_A$ using the Bertini theorems of Murthy 
\cite[Theorem~2.3]{Murthy} and Swan \cite[Theorem~1.3]{Swan}.
Let $I$ be an ideal of $A$ of height $d$ such that
${I}/{I^2}$ is generated by $d$ elements.
Let $\{\fm_1, \cdots , \fm_r\}$ be a set of smooth maximal ideals of $A$.
A special case of the Murthy-Swan Bertini theorem says that there exists an 
ideal $J \subset A$ (called a residual ideal of $I$)
such that the following hold (see \cite[Remarks~2.8, 3.2]{Murthy}). 
\begin{enumerate}
\item
$IJ$ is generated by $d$ elements.
\item
$I + J = \fm_i + J = A$ for $1 \le i \le r$.
\item
$J$ is a product of distinct smooth maximal ideals of $A$ of height $d$.
\end{enumerate}

It follows from (1) and (2) that $[I] + [J] = [IJ] = 0$ in $E_0(A)$. 
In particular,  $[I] = - [J]$ and (3) shows that $[J] \in E_s(A)$. 
This shows that $\phi_A$ is surjective.

To show that $\phi_A$ is injective, let $\alpha \in E_s(A)$ be such that
$\phi_A(\alpha) = 0$. 
By repeatedly applying the above Bertini theorem again,
we can write $\alpha = [J]$, where $J$ is a product of distinct smooth maximal
ideals of height $d$ in $A$. 
It follows from \cite[Theorem~4.2]{BS-3}
and \lemref{lem:ECG-SW} that $J$ is generated by $d$ elements.
On the other hand, $J$ is a product of distinct smooth maximal
ideals of height $d$ in $A$, and hence it is already a local complete 
intersection in $A$. We conclude that $J$ is a complete intersection in $A$
and hence $[J] = 0$ in $E_s(A)$.
\end{proof}

\vskip .3cm

The weak Euler class group $E_0(A)$ usually contains classes of ideals
$J \subset A$ which may be partly supported on the singular locus of
$\Spec(A)$. This presents an obstruction to the construction of a
direct homomorphism from $E_0(A)$ to the Chow group of $A$ or to $K_0(A)$. 
The modified group
$E_s(A)$ allows us to circumvent this obstruction as follows.

Let $A$ be a reduced affine algebra of dimension $d \ge 2$
over an algebraically closed field $k$. The assignment
$\fm \mapsto [A/{\fm}] \in \CH^d(A)$ induces a canonical group homomorphism
$G_s(A) \to \CH^d(A)$.
It follows from \cite[Lemma~2.2]{LW} that this assignment kills the classes of 
complete intersection ideals. In particular, it defines a
group homomorphism 
\begin{equation}\label{eqn:Euler-Chow*}
\theta_A \colon E_s(A) \to \CH^d(A). 
\end{equation}

The composite map $E_s(A) \xrightarrow{\theta_A} \CH^d(A) \xrightarrow{cyc_A}
K_0(A)$ (see ~\eqref{eqn:C-class}) takes $\fm$ to the class of the 
$A$-module $[A/{\fm}] \in K_0(A)$. If $I = \cap_i \fm_i$ is a regular ideal
of $A$, then the $A$-module $A/I$ has a class in $K_0(A)$.
Moreover, the Chinese remainder Theorem shows that
$[A/I] = \sum_i[A/{\fm_i}]$. It follows that $cyc_A \circ \theta_A([I])
= [A/I] \in K_0(A)$. We let $\gamma_A\colon E_s(A) \to K_0(A)$ denote the
composite map $cyc_A \circ \theta_A$.

The following result is a special case of \cite[Theorem~7.5]{GK}. We reproduce
the proof for the sake of completeness of our argument of the main
results of this section.
This result is classical for smooth
schemes (see \cite[\S~4.3]{Grothendieck}) and a proof of the general case is 
given in an old unpublished 
manuscript \cite{Levine-S} of Levine. 
The dimension two case of Levine's proof is available in
\cite{BS-*}. This result is used in \cite{Levine-2} to prove 
the prime-to-characteristic part of Murthy's conjecture. 
It is also used in \cite{Srinivas} and \cite{KSri}
to prove Murthy's conjecture for normal affine varieties.

\begin{thm}$($\cite[Theorem~7.5]{GK}$)$\label{thm:RR}
Let $A$ be a reduced affine algebra of dimension $d \ge 2$
over an algebraically closed field $k$.
Then the kernel of the cycle class map $cyc_X\colon \CH^d(A) \to K_0(A)$ is a
torsion group of exponent $(d-1)!$.
\end{thm}
\begin{proof}
Since the map $\theta_A\colon E_s(A) \to \CH^d(A)$ is evidently surjective, it
suffices to show that ${\rm Ker}(\gamma_A)$ is torsion of exponent $(d-1)!$.
So let $\alpha \in E_s(A)$ be such that $\gamma_A(\alpha) = 0$. By repeatedly
applying the Murthy-Swan Bertini theorem as in the proof of
\lemref{lem:MS-Bertini}, we can assume that $\alpha = [I]$, where $I \subset A$
is a regular ideal of height $d$. Our assumption then says that
$[A/I] = \gamma_A([I]) = 0$. Since $I$ is supported on the Cohen-Macaulay
(in fact, on the regular) locus of $A$, \cite[Lemma~1.2]{Mandal}
shows precisely that there is an $A$-regular sequence $\{f_1, \ldots , f_d\}$
in $I$ such that $I = (f_1, \ldots , f_d) + I^2$. 

If we let $J = (f_1, \ldots ,f_{d-1}) + I^{(d-1)!}$, then it is elementary
to check that $\sqrt{J} = \sqrt{I}$ and
$J = (f_1, \ldots , f_{d-1}, f^{(d-1)!}_d) + J^2$. 
In particular, $J$ has a class $[J] \in E_0(A)$.
Moreover, we have $[J] = (d-1)![I] \in E_0(A)$ by \cite[Lemma~4.1]{Das-Mandal}.
If we can show that $[J] = 0$ in $E_0(A)$, it will follow that
$(d-1)! [I] = 0$ in $E_0(A)$. We can then conclude from 
\lemref{lem:MS-Bertini} that $(d-1)![I] = 0$ is $E_s(A)$. 
We therefore need to show that $[J] = 0$ in $E_0(A)$.

Now, given our choice of $I$ and $J$, \cite[Theorem~2.2]{Murthy} says 
that there exists a projective $A$-module $P$ of rank $d$ such that
$P \surj J$ and $[P] - [A^d] = -[A/I] = - \gamma_A([I])$ in $K_0(A)$. 
As $\gamma_A([I]) = 0$, this implies that 
$P$ is stably free. Since $k$ is algebraically closed,
we can now apply Suslin's cancellation theorem \cite[Theorem~6]{Suslin}
to conclude that $P$ is free. In particular, $J$ is generated by
$d$ elements and so its class $[J]$ dies in $E_0(A)$, as desired.
We also note that \cite[Theorem~2.2]{Murthy} gives another
proof that $J/{J^2}$ is generated by $d$ elements.  
\end{proof}

\subsection{Applications}\label{sec:Appl-Euler-class}
We can now prove Murthy's conjecture and give other applications of
\thmref{thm:Tor-High-D}.

\begin{cor}\label{cor:Murthy*}
Let $A$ be a reduced affine algebra of dimension $d \ge 2$ over
an algebraically closed field $k$. 
Then the cycle class map $cyc_A\colon \CH^d(A) \to F^dK_0(A)$ is an isomorphism.
In particular, $F^dK_0(A)$ is uniquely divisible.
\end{cor}
\begin{proof}
The map $cyc_A$ is surjective by definition, and its injectivity is 
a straightforward consequence of Theorems~\ref{thm:Tor-High-D} and
~\ref{thm:RR}.
\end{proof}

The following is another application of \thmref{thm:Tor-High-D}. 
When the ring $A$ is smooth,
this was earlier proven by Bhatwadekar-R. Sridharan  
\cite[Corollary~4.15]{BS-1}.

\begin{thm}\label{thm:EGG-Chow}
Let $A$ be a reduced affine algebra of dimension $d \ge 2$
over an algebraically closed field $k$. Then the assignment
$\fm \mapsto [A/{\fm}] \in \CH^d(A)$ on the set of smooth maximal ideals
induces a canonical isomorphism
\[
\theta_A: E(A) \xrightarrow{\simeq} \CH^d(A).
\]
\end{thm}
\begin{proof}
In view of Lemmas~\ref{lem:ECG-SW} and ~\ref{lem:MS-Bertini}, it suffices
to show that the map $\theta_A\colon E_s(A) \to \CH^d(A)$ is an isomorphism.

The map $\theta_A$ is surjective by the definition of $\CH^d(A)$.
To prove its injectivity, let $\alpha \in E_s(A)$ be such that
$\phi_A(\alpha) = 0$. Using Murthy-Swan Bertini theorem as before, 
we can write $\alpha = [J]$, where $J$ is a 
product of distinct smooth maximal ideals of height $d$ in $A$. 
In particular, $J$ is a local complete intersection in $A$.
We have the maps $E_s(A) \xrightarrow{\theta_A} \CH^d(A) 
\xrightarrow{cyc_A} F^dK_0(A)$, where
the composite map takes $[J]$ to the class of $A/J$ in $F^dK_0(A)$
(see \cite[Corollary~2.7]{Murthy}).
Then $\theta_A([J]) = 0$ implies that the class $[A/J]$ is zero in 
$F^dK_0(A)$. Since $F^dK_0(A)$ is torsion-free by 
\corref{cor:Murthy*}, we can thus apply \cite[Corollary~3.4]{Murthy} to
conclude that $J$ is a complete intersection in $A$.
Equivalently, $[J] = 0$ in $E_s(A)$. This shows that $\theta_A$
is an isomorphism.
\end{proof}

\vskip .3cm

\subsection{Proof of \corref{cor:Intro-Res-5}}\label{sec:BBC*}
It follows from \thmref{thm:EGG-Chow} and \cite[Theorems~6.4.1, 6.4.2]{KSri-1}
that in both cases, one has $E_0(A) = 0$. 
Suppose now that $I \subset A$ is a local complete intersection ideal of
height $d$. Then ${I}/{I^2}$ is generated by $d$ elements as
an $A/I$-module. In particular, $I$ defines a class in $E_0(A)$, which is
zero. It follows from \cite[Theorem~4.2]{BS-3}
and \lemref{lem:ECG-SW} that $I$ is generated by $d$ elements and hence 
is a complete intersection. The last part of the corollary is obvious because
a product of smooth maximal ideals is a local complete intersection.
$\hfill \square$


\section{The strong Bloch-Srinivas conjecture}\label{sec:Pf:Chow-res}
We let $k$ be an algebraically closed field of
characteristic $p > 0$.
In this section, we prove a stronger form of the 
Bloch-Srinivas conjecture: \thmref{thm:Chowgroup}.
We shall then apply this to prove \corref{cor:BSC-0-1}.

\subsection{Proof of \thmref{thm:Chowgroup}}\label{sec:Pf-th-1.7}
Let $X$ be an affine or a projective variety over $k$ which has only
only isolated Cohen-Macaulay singularities.
Let $\pi\colon \wt{X} \to X$ be a resolution of singularities as in
\thmref{thm:Chowgroup} and let $E \subset \wt{X}$ be the reduced part of
the exceptional divisor.

We first show that the cycle class map (see ~\eqref{eqn:C-class})
$cyc_X\colon\CH^d(X) \surj F^dK_0(X)$ is an isomorphism.
In view of \corref{cor:Murthy*}, we can assume $X$ to be projective.
Since $\dim(X) \ge 2$ and it has isolated Cohen-Macaulay singularities,
$X$ must be normal. It follows from 
\cite[Chap.~II, Theorem~11]{Lang} that the map ${\rm Alb}(X) \to 
{\rm Alb}(\wt{X})$ is an isomorphism, where ${\rm Alb}(X)$ is the
Albanese variety of a normal projective variety in the sense of
Lang \cite[Chap.~III, \S~3]{Lang}. In particular, it follows from
\cite[Theorem~1.6]{KSri} that there is a commutative diagram
\[
\xymatrix@C.6pc{
\CH^d(X) \ar[r] \ar[d] & F^dK_0(X) \ar[d] \\
\CH^d(\wt{X}) \ar[r] & F^dK_0(\wt{X})}
\]
in which the left vertical arrow is an isomorphism on torsion.
Since the bottom horizontal arrow is known to be an isomorphism
(see \cite[Theorem~3.2]{Levine-2}),
it follows that the top horizontal arrow must be injective on the
torsion subgroup. On the other hand, \cite[Theorem~3.2]{Levine-2} also says
that the top horizontal arrow has torsion kernel.
We conclude that ${\rm Ker}(cyc_X) = 0$.

We now prove the second isomorphism of the theorem, assuming
$X$ is affine or projective over $k$.
Let $Y$ denote the singular locus of $X$ with the reduced induced closed
subscheme structure.
It follows from \cite[Lemma~3.1]{Krishna-3} that the map
$F^dK_0(X, Y) \to F^dK_0(X)$ is an isomorphism.
We thus have a commutative diagram
\begin{equation}\label{eqn:Chowgroup-0}
\xymatrix@C1pc{
F^dK_0(X, Y) \ar[r]^-{\simeq} \ar[d]_{\pi^*} & 
F^dK_0(X) \ar[d]^{\pi^*} \ar@{-->}[dl] \\
F^dK_0(\wt{X}, E) \ar[r] & F^dK_0(\wt{X}).}
\end{equation}

In order to complete the proof of the theorem,
we are therefore left with showing that the map
$\pi^*\colon F^dK_0(X, Y) \to F^dK_0(\wt{X}, E)$ is an isomorphism.
Since this is surjective by the definition of $F^dK_0(\wt{X}, E)$,
the main point is to show that this map is injective as well.

Let $F$ denote the kernel of the map $F^dK_0(X, Y) \to F^dK_0(\wt{X}, E)$
and consider the commutative diagram with exact top row:
\begin{equation}\label{eqn:Chowgroup-1}
\xymatrix@C.6pc{
0 \to F \ar[r] \ar[dr] & F^dK_0(X,Y) \ar[r] \ar[d]^{\pi^*} & 
F^dK_0(\wt{X}, E) \ar[r] \ar[d]^{\iota_*} & 0 \\
& F^dK_0(\wt{X}) \ar@{=}[r] & F^dK_0(\wt{X}).}
\end{equation}

Since $X$ is quasi-projective over $k$, we can apply \cite[Lemma~1.4]{Levine-5} 
to find a map $X' \xrightarrow{\pi'} \wt{X} \xrightarrow{\pi} X$ such that
$X'$ is the blow-up of $X$ along a closed subscheme $Z \subset X$ with
$Z_{\rm red} = Y$. Note that most of the results of {\sl loc. cit.}
are valid only in characteristic zero, but the above cited result
is characteristic-free.  Set $E' = \pi'^{-1}(E)$
with the reduced structure.

It follows from \propref{prop:pro-desc-CM-A} that the kernel of the 
composite map $F^dK_0(X, Y) \to F^dK_0(\wt{X}, E) \to F^dK_0(X', E')$
is a $p$-primary torsion group of bounded exponent. 
Since $F$ lies in this kernel, we conclude from ~\eqref{eqn:Chowgroup-1}
that $F$ is a subgroup of ${\rm Ker}(\pi^*)$ and is $p$-primary torsion 
of bounded exponent.

On the other hand, if $X$ is affine, it follows from 
\cite[Corollary~1.7]{KSri} and the isomorphisms 
$\CH^d(X) \simeq F^dK_0(X) \simeq F^dK_0(X,Y)$ (shown above) that $F^dK_0(X,Y)$
is torsion-free. If $X$ is projective, these isomorphisms,
combined with the isomorphism $\CH^d(\wt{X}) \simeq F^dK_0(\wt{X})$
(see \cite[Theorem~3.2]{Levine-2}) and \cite[Theorem~1.6]{KSri},
imply that ${\rm Ker}(\pi^*)$ is torsion-free. We must therefore have $F = 0$.
This finishes the proof.
$\hfill \square$

\subsection{Proof of \corref{cor:BSC-0-1}}\label{sec:SBC-**}
Set $X = \Spec(A)$ and
let $\pi\colon \wt{X} \to X$ denote the blow-up of $X$ at its vertex $P$.
It is easy to check that there is a projection map $p\colon \wt{X} \to Z$ 
which is an affine bundle of rank one. 
Moreover, the 0-section of this affine bundle $\iota\colon Z \inj \wt{X}$
is also the exceptional divisor for $\pi$.
It particular, $\pi$ is a good resolution of singularities of $X$ with
reduced exceptional divisor $Z$.

The homotopy invariance implies that the 
pull-back map $\iota^*\colon K(\wt{X}) \to K(Z)$ is a homotopy 
equivalence of spectra. Equivalently, $K(\wt{X}, Z)$ is contractible.
In particular,  $F^dK_0(\wt{X}, Z) =0$.
It follows from \thmref{thm:Chowgroup} that $\CH^d(X) = 0$.
The last part of the corollary follows from the vanishing of
$\CH^d(X)$ and \cite[Theorem~3.7]{Murthy}.
$\hfill \square$

\bigskip

\noindent\emph{Acknowledgments.}
The author would like to thank M. K. Das for help in locating 
a reference for the relation between the Euler and weak Euler class groups.
He would like to thank Marc Levine for encouragement and,
S. M. Bhatwadekar and Chuck Weibel for useful comments
on this paper. The author is grateful to the
referee for many comments and suggestions which greatly improved the 
exposition of this paper.

\end{document}